\documentclass[5p, a4paper]{article}

\usepackage[utf8]{inputenc}
\usepackage[T1]{fontenc}
\usepackage{amssymb}
\usepackage{amsfonts}
\usepackage{amsmath}
\usepackage{amsthm}
\usepackage{graphicx}
\usepackage{tikz}
\usepackage{pgfplots}
\usepackage{subfig} 
\usepackage{textcomp}
\usepackage{mathtools}
\usepackage{fullpage}

\usepackage{moreverb}
\usepackage{enumerate}

\newcommand{\R}{\mathbb{R}}
\newcommand{\N}{\mathbb{N}}

\newcommand{\A}{\mathcal{A}}
\newcommand{\C}{\mathcal{C}}

\newcommand{\eremk}{\hbox{}\hfill\rule{0.8ex}{0.8ex}}

\newtheorem{definition}{Definition}[section]
\newtheorem{theorem}[definition]{Theorem}
\newtheorem{lemma}[definition]{Lemma}

\newtheorem{corollary}[definition]{Corollary}
\newtheorem{remark}[definition]{Remark}

\newtheorem{assumption}[definition]{Assumption}

\numberwithin{equation}{section}

\newcommand{\diam}{\operatorname*{diam}}

\newcommand{\abs}[1]{\left\vert #1 \right\vert}
\newcommand{\skp}[1]{\left< #1 \right>}
\newcommand{\norm}[1]{\left\| #1 \right\|}

\newcommand{\T}{\mathcal{T}}
\newcommand{\dist}{\operatorname*{dist}}

\newcommand{\supp}{\operatorname*{supp}}

\definecolor{dgreen}{rgb}{0,0.65,0}
\definecolor{ddgreen}{rgb}{0,0.4,0.4}

\begin{document}

\hskip 10 pt

\begin{center}
{\fontsize{14}{20}\bf 
Local convergence of the FEM for the integral fractional Laplacian}
\end{center}

\begin{center}
\bigskip
\textbf{Markus Faustmann\footnote{
Institute for Analysis and Scientific Computing, TU Wien, Vienna, Austria,
\tt{markus.faustmann@tuwien.ac.at}}, 
Michael Karkulik\footnote{
Departamento de Matem\'atica, Universidad T\'ecnica Federico Santa Mar\'ia, Valpara\'iso, Chile,
\tt{michael.karkulik@usm.cl}}, 
Jens Markus Melenk\footnote{
Institute for Analysis and Scientific Computing, TU Wien, Vienna, Austria,
\tt{melenk@tuwien.ac.at}}}\\
\bigskip
\end{center}

\begin{abstract}
For first order discretizations of the integral fractional Laplacian we provide sharp local error estimates 
on proper subdomains in both the local $H^1$-norm and the localized energy norm. Our estimates have the form of a local best approximation error plus a global
error measured in a weaker norm.
\end{abstract}

\section{Introduction}
\medskip

It is well-known that the rate of convergence of the finite element method (FEM) applied to elliptic PDEs
depends on the global regularity of the sought solution. However, if the quantity of interest is just the
error on some subdomain, one could hope that - provided the solution is smoother locally - the error
decays faster. This is indeed the case, and the proof of this observation goes back at least to the work \cite{NS74}. Since then, the local
behavior of FEM approximations has been well understood and various refinements of the original arguments can be found 
in, e.g., \cite{wahlbin91,DGS11}. In these works
the locality of the differential operator is used to prove estimates of the form that  the local error is bounded by
a local best approximation and a local error in a weaker norm.

Currently, models of anomalous diffusion are studied in various applications, which gives rise to fractional PDEs, i.e.,
fractional powers of elliptic operators. 
The numerical approximation of fractional PDEs by the finite element method, as studied here, is an active research field, and 
we mention, e.g., \cite{NOS15,AB17,ABH18,BMNOSS19,BC19} for global {\textit a priori} error analyses.  
For other numerical methods applied to fractional PDEs, we refer to \cite{BP15,BLP19} for a semigroup approach, 
\cite{SXK17,AG18} for techniques that exploit eigenfunction expansions,
as well as the survey articles, \cite{BBNOS18,LPGSGZMCMAK18}.

In comparison to integer order elliptic operators such as the
Laplacian, dealing with the fractional version is more challenging due to the non-local nature of
fractional operators. In this regard, fractional operators are similar to the integral operators appearing in the boundary element method (BEM), \cite{SauterSchwab}. For the BEM, local error estimates and improved
convergence results are available as well, see, e.g., \cite{saranen87,Tran95,ST96,FM18}, which differ from the ones for
the FEM in the way that the error contribution in the weaker norm -- sometimes called 'slush term' in the
literature -- is in a global norm instead of a local norm due to the non-local nature of the appearing
operators.

In this article, we provide local error estimates for the FEM applied to the integral fractional Laplacian $(-\Delta)^s$ for $s \in (0,1)$ of the form
\begin{align}\label{eq:mainresult}
\norm{u-u_h}_{H^1(\Omega_0)} \leq C \left(\inf_{w_h\in V_h}\norm{u-w_h}_{H^1(\Omega_1)}+ \norm{u-u_h}_{H^{s-1/2}(\Omega)}\right),
\end{align}
where $\Omega_0 \subset \subset \Omega_1$ are open subsets of the computational domain, $V_h$ is the finite element space, $u$ denotes the exact solution of the fractional differential equation, and $u_h \in V_h$ is its Galerkin approximation.
A direct consequence of this estimate and a duality argument is that the FEM converges locally in
the $H^1$-norm with order $1-\varepsilon$ for any $s \in (0,1)$ and $\varepsilon >0$, provided the solution has $H^2$-regularity locally, and the typical global regularity  $u \in H^{1/2+s-\varepsilon}(\Omega)$. In contrast, global convergence in the $H^1$-norm can only be expected for $s \in (1/2,1)$ and then is limited to the rate $s-1/2-\varepsilon$, see \cite{BC19}.
Our primary focus are meshes that are quasi-uniform in the region of interest $\Omega_1$. This class of meshes particularly includes meshes 
  that are graded towards the boundary. Generalizations to meshes that are locally refined in $\Omega_1$ are possible.
We briefly discuss these issues in Remarks~\ref{rem:gradedH1}, \ref{rem:graded}, and \ref{remk:3.9}.

Recently and independently a local error analysis similar to ours was derived in \cite{BLN20} using different techniques. Our result differs from the estimates \cite{BLN20} in two ways: First, while \cite{BLN20} provides local estimates in the energy norm, we additionally study the stronger local $H^1$-norm. Second, the slush term in \cite{BLN20} is in a different norm, the $L^2$-norm, whereas we obtain the $H^{s-1/2}$-norm. For $s<1/2$ this gives a stronger estimate, but for $s>1/2$ a weaker estimate.
  With our techniques the slush term could also be weakened to even weaker norms (such as the $L^2$-norm for $s>1/2$). 
However, we chose the $H^{s-1/2}$-norm for the slush term since for quasi-uniform meshes the use of weaker norms would not give 
better convergence rates due to the limited regularity of the pertinent dual problem.  

The paper is structured as follows: Section~\ref{sec:mainresults} provides the model problem, the discretization by a
lowest order Galerkin method and the main result, Theorem~\ref{thm:local}, which presents the local error estimate
in the $H^1$-norm, \eqref{eq:mainresult}. If the solution is smoother locally, the improvement in the local convergence rates
are stated in Corollary~\ref{cor:local} for the case of quasi-uniform meshes.

Section~\ref{sec:proofs} provides the proofs of the main results as well as the corresponding result for the energy
norm. Finally, the numerical examples in Section~\ref{sec:numerics} confirm the sharpness of the theoretical local convergence rates of our main result. \\

Concerning notation: For bounded, open sets $\omega \subset \R^d$, integer order Sobolev 
spaces $H^t(\omega)$, $t \in \N_0$, are  defined in the usual way. For $t \in (0,1)$, 
fractional Sobolev spaces are given in terms of the seminorm 
$|\cdot|_{H^t(\omega)}$ and the full norm $\|\cdot\|_{H^t(\omega)}$ by 
\begin{equation}
\label{eq:norm} 
|v|^2_{H^t(\omega)} = \int_{x \in \omega} \int_{y \in \omega} \frac{|v(x) - v(y)|^2}{\abs{x-y}^{d+2t}}
\,dx\,dy, 
\qquad \|v\|^2_{H^t(\omega)} = \|v\|^2_{L^2(\omega)} + |v|^2_{H^t(\omega)}, 
\end{equation}
where we denote the Euclidean norm in $\R^d$ by $\abs{\;\cdot\;}$. 
Moreover, for bounded Lipschitz domains $\Omega \subset \R^d$, we define the spaces
\begin{align*}
 \widetilde{H}^{t}(\Omega) := \{u \in H^t(\R^d) \,: \, u\equiv 0 \; \text{on} \; \R^d \backslash \overline{\Omega} \}
\end{align*}
of $H^t$-functions with zero extension, equipped with the norm
\begin{align*}
 \norm{v}_{\widetilde{H}^{t}(\Omega)}^2 := \norm{v}_{H^t(\Omega)}^2 + \norm{v/\rho^t}_{L^2(\Omega)}^2,
\end{align*}
where $\rho(x)$ is the distance of a point $x \in \Omega$ from the boundary $\partial \Omega$.
An equivalent norm is the full space norm $H^t(\R^d)$ of the zero extension of $u$. Throughout this work,
we will frequently view functions in $\widetilde{H}^t(\Omega)$ as elements of $H^t(\R^d)$ through the zero extension.
For $t \in (0,1)\backslash \{\frac 1 2\}$, the norms $\norm{\cdot}_{\widetilde{H}^{t}(\Omega)}$  and $\norm{\cdot}_{H^{t}(\Omega)}$ are equivalent, \cite{grisvard11}.
Furthermore, for $t > 0$, we denote by $H^{-t}(\Omega)$ the dual space of $\widetilde{H}^t(\Omega)$ and
by $\widetilde{H}^{-t}(\Omega)$ the dual space of $H^t(\Omega)$.
For $t \in \R$, we denote by $H^t_{loc}(\R^d)$ the distributions on $\R^d$, whose restriction to any ball $B_R(0)$ is in $H^t(B_R(0))$.
As usual, we write $\skp{\cdot,\cdot}_{L^2(\Omega)}$ for the duality pairing that extends the $L^2(\Omega)$-inner product.
\bigskip

We note that there are several different ways to define the fractional Laplacian $(-\Delta)^s$. A classical 
definition on the full space ${\mathbb R}^d$ is in terms of the Fourier transformation ${\mathcal F}$, 
i.e., $({\mathcal F} (-\Delta)^s u)(\xi) = |\xi|^{2s} ({\mathcal F} u)(\xi)$. 
A consequence of 
this definition is the mapping property, (see, e.g., \cite{BBNOS18})
\begin{equation}
\label{eq:mapping-property}
(-\Delta)^s: H^t(\R^d) \rightarrow H^{t-2s}(\R^d), \qquad t \ge s, 
\end{equation}
where the Sobolev spaces $H^t(\R^d)$, $t \in \R$, are defined in terms of the Fourier transformation,
\cite[(3.21)]{mclean00}. 
Alternative, equivalent definitions of $(-\Delta)^s$ exist, e.g., 
via spectral, semi-group, or operator theory,~\cite{Kwasnicki} or 
via singular integrals.
Specifically, the integral fractional Laplacian can alternatively be defined 
pointwise for sufficiently smooth functions $u$ as the principal value integral  
\begin{align}\label{eq:fracLaplaceDef}
(-\Delta)^su(x) := C(d,s) \; \text{P.V.} \int_{\R^d}\frac{u(x)-u(y)}{\abs{x-y}^{d+2s}} \, dy \quad \text{with} \quad
C(d,s):= - 2^{2s}\frac{\Gamma(s+d/2)}{\pi^{d/2}\Gamma(-s)},
\end{align}
where $\Gamma(\cdot)$ denotes the Gamma function. 
%
\section{Main Results}\label{sec:mainresults}
\subsection{The model problem}

Let $\Omega \subset \R^d$ be a bounded domain. We consider the fractional differential equation 
\begin{subequations}\label{eq:modelproblem}
\begin{align}
 (-\Delta)^su &= f \qquad \text{in}\, \Omega, \\
 u &= 0 \quad \quad\, \text{in}\, \Omega^c:=\R^d \backslash \overline{\Omega},
\end{align}
\end{subequations}
where $s \in (0,1)$ and $f \in H^{-s}(\Omega)$ is a given right-hand side. 
Equation (\ref{eq:modelproblem}) is understood as in weak form: Find 
$u \in \widetilde{H}^s(\Omega)$ such that 
\begin{equation}
\label{eq:weakform}
a(u,v):= \skp{(-\Delta)^s u,v}_{L^2(\R^d)} = \skp{f,v}_{L^2(\Omega)} 
\qquad \forall v \in \widetilde{H}^s(\Omega). 
\end{equation}
The bilinear form $a$ has the alternative representation (cf., e.g.,   
\cite[Thm.~{1.1 (e),(g)}]{Kwasnicki}) 
\begin{equation}
a(u,v) = 
 \frac{C(d,s)}{2} \!\! \int\int_{\R^d\times\R^d} \!\!\!
 \frac{(u(x)-u(y))(v(x)-v(y))}{\abs{x-y}^{d+2s}} \, dx \, dy 
\quad \forall u,v \in \widetilde{H}^s(\Omega). 
\end{equation}
Existence and uniqueness of $u \in \widetilde{H}^s(\Omega)$ follow from
the Lax--Milgram Lemma for any $f \in H^{-s}(\Omega)$. 
The bilinear form $a$ induces an invertible operator $\A: \widetilde{H}^s(\Omega) \rightarrow H^{-s}(\Omega)$.

Our analysis hinges on the regularity pickup of certain
dual problems. We formulate this as an assumption:

\begin{assumption}\label{ass:shift}
For the domain $\Omega \subset \R^d$ 
and some $0< \varepsilon < s/2$, there holds 
the shift theorem for $\A$: 
\begin{align*}
f \in H^{1/2-s-\varepsilon}(\Omega) \quad \implies \quad u  = \A^{-1} f \in \widetilde{H}^{1/2+s-\varepsilon}(\Omega)
\end{align*}
together with the {\sl a priori} estimate 
\begin{align*}
\norm{u}_{\widetilde{H}^{1/2+s-\varepsilon}(\Omega)} \leq C \norm{f}_{H^{1/2-s-\varepsilon}(\Omega)}. 
\end{align*} 
The constant $C>0$ depends only on $\Omega$, $d$, $s$, and $\varepsilon$.
\end{assumption}

\begin{remark}
\label{remk:shift-theorem}
For Lipschitz domains, Assumption~\ref{ass:shift} holds for any $\varepsilon >0$ by \cite[Thm.~{1.2}, Cor.~{1.3}]{BN21}; the earlier result by 
\cite{Grubb15} showed Assumption~\ref{ass:shift} for any $\varepsilon >0$
and smooth domains.
\eremk
\end{remark}

\subsection{Discretization}
We consider a regular triangulation $\T_h$ (in the sense of Ciarlet, \cite{ciarlet78}) 
of $\Omega$ consisting of open simplices that is also $\gamma$-shape regular in the sense 
\begin{align*}
\max_{T \in \T_h} \big( \diam(T) / |T|^{1/d} \big) \le \gamma < \infty.
\end{align*} 
Here, $\diam(T):=\sup_{x,y\in T}|x-y| =:h_T$ denotes the Euclidean diameter of $T$, whereas $|T|$ is the $d$-dimensional Lebesgue volume. 
The mesh width of $\T_h$ is $h:=\max_{T \in \T_h} h_T$. 

For an element $T \in \T_h$, we define the element patch as
\begin{align*}
 \omega_T := \operatorname*{interior}\Big(\bigcup_{T' \in \T(T)}\overline{T'}\Big) \quad 
 \text{with} \quad  \T(T):=\{T'\in\T_h \,:\, \overline{T'}\cap\overline{T} \neq \emptyset\}.
\end{align*}
Similarly, for a function $\eta \in C_0^\infty(\R^d)$, we write 
\begin{align}
 \omega_\eta := \operatorname*{interior}\Big(\bigcup_{T' \in \T(\eta)}\overline{T'}\Big) \quad 
 \text{with} \quad \T(\eta):=\{T'\in\T_h \,:\, \overline{T'}\cap\operatorname*{supp}\eta \neq \emptyset\},\\
 \omega_\eta^2 := \operatorname*{interior}\Big(\bigcup_{T' \in \T(\eta)^2}\overline{T'}\Big) \quad 
 \text{with} \quad \T(\eta)^2:=\{T'\in\T_h \,:\, \overline{T'}\cap\overline{\omega_\eta} \neq \emptyset\}
\end{align}
for layers of elements of the triangulation surrounding the support of the function $\eta$.

To discretize~\eqref{eq:weakform}, we consider the Galerkin method with piecewise linears.
More precisely, for $T \in \T_h$, we denote the space of all affine functions on $T$ by  $\mathcal{P}^1(T)$.
The spaces of $\T_h$-piecewise affine and globally continuous functions are then defined as
\begin{align*}
S^{1,1}(\T_h) &:= \{u \in H^1(\Omega) \,:\, u|_{T} \in \mathcal{P}^1(T) \text{ for all } T \in \T_h\}
\quad \text{and} \\
S^{1,1}_0(\T_h)& := S^{1,1}(\T_h) \cap \widetilde H^1(\Omega).
\end{align*}
Using $S^{1,1}_0(\T_h) \subset \widetilde{H}^s(\Omega)$ as ansatz  and test space, we seek
a finite element solution $u_h \in S^{1,1}_0(\T_h)$ such that
\begin{align}\label{eq:galerkin}
 a(u_h, v_h) = \skp{f,v_h}_{L^2(\Omega)} 
 \quad \text{for all } v_h \in S^{1,1}_0(\T_h).
\end{align}
The Lax--Milgram Lemma provides unique solvability of \eqref{eq:galerkin}.

\subsection{The main results}
The following theorem is the main result of this article. It estimates the 
local FEM-error in the $H^1$-norm and the energy-norm by the local best approximation and a global error in a weaker norm.

We consider meshes that are quasi-uniform on the domain of interest $\Omega_1 \subset \Omega$. More precisely, for a given triangulation $\T_h$, we introduce 
the triangulation 
\begin{align}\label{eq:triangulationO1}
\T_{h,1} := \left\{ T \in \T_h \,:\, \overline{T}\cap\Omega_1\neq \emptyset \right\},
\end{align}
of a suitable superset of $\Omega_1$.  We assume that $\T_{h,1}$ is quasi-uniform, i.e., there is a  constant $C_0>0$ such that
  $h_{1,\min} \geq C_0 h_{1}$, where   $h_{1,\min} :=\min\left\{ h_T \,:\, T \in \T_{h,1} \right\}$   and $h_{1} :=\max\left\{ h_T \,:\, T\in \T_{h,1} \right\}$ denote
  the minimal and maximal element size on $\Omega_1$.

\begin{theorem}\label{thm:local}
  Let $R>0$ and $\Omega_0 \subset\Omega_1\subset\Omega$ be open sets with $\dist(\Omega_0,\partial\Omega_1)>R$ and $\dist(\Omega_1,\partial\Omega)>R$.
  Let $\T_h$ be a triangulation of $\Omega$ of mesh width $h$, and let $\T_{h,1}$ be defined by \eqref{eq:triangulationO1} with mesh width $h_1$.
   Suppose that
  $\frac{h_{1}}{R} \leq \frac{1}{40}$.
  Let Assumption~\ref{ass:shift} be valid, $u$ solve \eqref{eq:weakform} and $u_h$ be its Galerkin approximation solving \eqref{eq:galerkin}.
  
  Let $\eta_0$, $\eta_2$, $\eta_4  \in C_0^\infty(\R^d)$ be cut-off functions satisfying $\dist(\left\{ \eta_j\equiv 1 \right\},\partial\supp\eta_j)<\frac{R}{40}$ and \linebreak
$\|\nabla^k \eta_j\|_{L^\infty(\R^d)} \leq C_\eta R^{-k}$
  for $j\in\{0,2,4\}$, $k \in \{0,1,2\}$ and some fixed $C_\eta > 0$. Additionally
 suppose that $\eta_0\equiv 1$ on $\Omega_0$ and $\dist(\Omega_0,\partial\supp\eta_0)\leq \frac{R}{40}$,
  as well as
$\dist(\Omega_0,\partial\left\{ \eta_2\equiv 1 \right\})\geq 2R/10$, \linebreak $\dist(\left\{\supp \eta_2, \partial \{\eta_4\equiv 1\} \right\})\geq R/10+3R/40$, and  $\dist(\left\{\supp\eta_4, \partial\Omega_1 \right\})\geq 5R/10+3R/40$.

  Then, there are constants $C_1$, $C_2>0$ depending only on $\Omega$, $\Omega_0$, $\Omega_1$, $R$, $d$, $s$, 
$C_\eta$, and the $\gamma$-shape regularity of $\T_h$ such that
  the following holds:
\begin{enumerate}[(i)]
\item\label{item:thm:local-i}
\begin{align*}
 \norm{\eta_0(u-u_h)}_{\widetilde H^s(\Omega)} \! \leq C_1 \left( \operatorname*{inf}_{v_h \in S^{1,1}_0(\T_h)} \norm{\eta_2(u- v_h)}_{\widetilde H^s(\Omega)} \! + 
 \norm{u-u_h}_{H^{s-1/2}(\Omega)}\right).
\end{align*}

\item \label{item:thm:local-ii}
Assume $u \in H^1(\Omega_1)$. Then, 
\begin{align*}
 \norm{u-u_h}_{H^1(\Omega_0)} &\leq C_2 \Bigl[\, \operatorname*{inf}_{v_h \in S^{1,1}_0(\T_h)} \norm{u- v_h}_{H^1(\Omega_1)}  \\
\nonumber 
& \quad \qquad + 
 \left(\frac{h}{h_1}\right)^{1-s} \left( \|\eta_4( u - u_h)\|_{\widetilde{H}^s(\Omega)} + \norm{u-u_h}_{H^{s-1/2}(\Omega)}\right) \Bigr].
\end{align*}
\end{enumerate}
\end{theorem}

Assuming additional regularity for the solution locally, the following corollary provides 
optimal rates for the local FEM-error. 

\begin{corollary}\label{cor:local}
Let $\T_h$ be a quasi-uniform triangulation with mesh width $h$.
With the assumptions of Theorem~\ref{thm:local}, let $\Omega_1 \subset \Omega_2 \subset \Omega$ with 
$\dist(\Omega_1,\partial \Omega_2) > 0$. Let $\varepsilon>0$ be given by Assumption~\ref{ass:shift}.
\begin{itemize}
\item[(i)]
 Let  $u \in \widetilde H^{s+\alpha}(\Omega)\cap H^{s+\beta}(\Omega_2)$
with $0<\alpha,\beta$. Then,
 \begin{align*}
   \norm{\eta_0(u-u_h)}_{\widetilde H^s(\Omega)} \leq C h^{\min\{1/2+\alpha-\varepsilon,\beta,2-s\}}.
\end{align*}

\item[(ii)]  Let  $u \in \widetilde H^{s+\alpha}(\Omega)\cap H^{1+\beta}(\Omega_2)$
with $0<\alpha,\beta$. Then,
 \begin{align*}
   \norm{u-u_h}_{H^1(\Omega_0)} \leq C h^{\min\{1/2+\alpha-\varepsilon,\beta,1\}}.
\end{align*}
\end{itemize}
Here, the constants $C>0$ depend only on $\Omega,\Omega_0,\Omega_1,R,\dist(\Omega_1,\partial \Omega_2),d,s,\alpha,\beta,$ the $\gamma$-shape regularity of $\T_h$, and $\varepsilon$.
\end{corollary}

For sufficiently smooth right-hand sides $f$, 
solutions of \eqref{eq:weakform} can be expected to be in 
$H^{s+1/2-\varepsilon}(\Omega)$ for any $\varepsilon >0$ 
(cf.\ the shift theorem of Assumption~\ref{ass:shift}), which 
gives $\alpha = 1/2-\varepsilon$ in Corollary~\ref{cor:local}. However,
typically, solutions are smoother on any subdomain $\Omega_2 \subset \Omega$
(cf.\ Lemma~\ref{lem:intreg})
that satisfies $\dist(\Omega_2, \partial \Omega)>0$. For $u \in H^{2}(\Omega_2)$, i.e., $\beta = 1$ in the second statement of the
Corollary~\ref{cor:local},
this leads to convergence of order $\mathcal{O}(h^{1-2\varepsilon})$ in the $H^1$-norm locally.

\begin{remark}\label{rem:sharpness}
Corollary~\ref{cor:local} gives sharp local convergence results both in the $H^1$-norm and the energy norm on quasi-uniform meshes. For sufficiently high local regularity, both estimates give the same rate of convergence locally. 
This is due to the fact that in this case the slush term in Theorem~\ref{thm:local} dominates and both local error estimates employ the same slush term. 
However, weakening the norm of the slush term does not improve the convergence rates on quasi-uniform meshes, since the duality arguments used (see the proof of Corollary~\ref{cor:local}) to estimate the $H^{s-1/2}(\Omega)$-norm already exploits the maximal regularity of the dual problem available. Improved rates may be expected, however,  if meshes that are graded towards $\partial\Omega$ are employed.

We also mention that local estimates in the $L^2$-norm are possible, but, for the same reason, the rate of convergence locally in $L^2$ for locally smooth solutions is not better than in the energy or $H^1$-norm. 
We refer to the numerical results in Section~\ref{sec:numerics} for the sharpness of these observations.
\eremk
\end{remark}

\begin{remark}\label{rem:gradedH1}
In order to counteract the singular behavior of solutions of fractional PDEs near the boundary,
meshes that are graded towards the boundary can be employed, cf.~\cite{AB17,BLN20}.
Due to our assumption $\dist(\Omega_1,\partial\Omega)>R$ (for fixed $R>0$), our local estimates of Theorem~\ref{thm:local} are also valid for such graded meshes.
\eremk
\end{remark}

\subsection{The fractional Laplacian and the Caffarelli-Silvestre extension}

A key tool in the proof of a similar result for the BEM in \cite{FM18} was the use of properties of 
the (single- or double-layer) potentials or, more precisely, a Caccioppoli type inequality. This interior regularity result allowed us to 
control derivatives of the potentials.

For the fractional Laplacian a similar idea can be employed, where the role of the potential is taken by 
the Caffarelli-Silvestre extension problem \cite{CafSil07}, i.e., the fractional Laplacian  
can be understood as a Dirichlet-to-Neumann operator of a degenerate elliptic PDE on a half space in 
$\R^{d+1}$: Given $v \in \widetilde{H}^s(\Omega)$, let $U = U(x,\mathcal{Y})$ 
solve for $\alpha = 1-2s \in (-1,1)$
\begin{subequations}\label{eq:extension}
\begin{align}
\label{eq:extension-a}
 \operatorname*{div} (\mathcal{Y}^\alpha \nabla U) &= 0  \;\quad\quad\text{in} \; \R^d \times (0,\infty), \\ 
\label{eq:extension-b}
 U(\cdot,0) & = v  \,\quad\quad\text{in} \; \R^d.
\end{align}
\end{subequations}
(The solution $U$ is unique by requiring $U$ to be in the Beppo-Levi space 
${\mathcal B}^1_\alpha(\R^d \times \R^+)$ introduced below.)
Then, the fractional Laplacian can be recovered as the Neumann data of the extension problem in the 
sense of distributions, \cite[Thm.~{3.1}]{CafSil07,cabre-sire14}:
\begin{align}
 -d_s \lim_{\mathcal{Y}\rightarrow 0^+} \mathcal{Y}^\alpha \partial_\mathcal{Y} U(x,\mathcal{Y}) = (-\Delta)^s v, 
\qquad 
d_s = 2^{2s-1}\Gamma(s)/\Gamma(1-s).
\end{align}
The natural Hilbert space for weak solutions of equation \eqref{eq:extension} is a weighted Sobolev space.
For measurable subsets $\omega \subset \R^d \times \R^+$, we 
define the weighted $L^2$-norm 
\begin{align*}
 \norm{U}_{L^2_\alpha(\omega)}^2 := \int_{\omega} \mathcal{Y}^{\alpha}\abs{U(x,\mathcal{Y})}^2 dx \, d\mathcal{Y} 
\end{align*}
and denote by $L^2_\alpha(\omega)$ the space of 
square-integrable functions with respect to the weight $\mathcal{Y}^\alpha$. 
The Caffarelli-Silvestre extension is conveniently described in terms of the Beppo-Levi space 
${\mathcal B}^1_{\alpha}(\R^d \times  \R^+):= \{U \in {\mathcal D}^\prime(\R^{d} \times \R^+)\,:\, 
\nabla U \in L^2_\alpha(\R^d \times \R^+)\}$. 
Elements of ${\mathcal B}^1_\alpha(\R^d \times \R^+)$ are  in fact in $L^2_{\rm loc}(\R^d \times \R^+)$
and
one can give meaning to their trace 
at $\mathcal{Y} = 0$, which is denoted $\operatorname{tr} U$. Recalling $\alpha = 1-2s$, one has in fact
$\operatorname{tr} U \in H^s_{\rm loc}(\R^d)$ (see, e.g., \cite{KarMel18}).

\section{Proof of Theorem~\ref{thm:local}}
\label{sec:proofs}

In order to make the proof of the main result more accessible, we sketch the main ingredients in the following, details 
are given in lemmas below. We also fix the notation for this section in the following listings.
\medskip

Throughout this section, we use the notation $\lesssim$ to abbreviate $\leq$ up to a generic constant $C>0$
that does not depend on critical parameters in our analysis  such as the (local) mesh width.  

\begin{enumerate}[(i)]
\item \label{def:cut-off} \emph{Localization with cut-off functions:}
    The assumptions of Theorem~\ref{thm:local} allow us to choose additional nested open sets $\Omega_{j/10}$ and cut-off functions 
$\eta_j$ for $j=1,\dots,9$ such that 
    \begin{align*}
      & \Omega_{j/10}\subset\Omega_{(j+1)/10} \qquad\text{ and } \qquad\dist(\Omega_{j/10},\partial\Omega_{(j+1)/10})\geq R/10, \\
      & \eta_j\equiv 1\text{ on } \Omega_{j/10} \qquad \text{ and } \qquad \dist(\partial\supp\eta_j,\Omega_{j/10}) < \frac{R}{40} 
    \end{align*}
    as well as 
    $\norm{\eta_j}_{W^{k,\infty}(\R^d)} \lesssim R^{-k}$ for $k\in\left\{ 0,1,2 \right\}$. We note that these assumptions together with $\frac{h_1}{R}<\frac{1}{40}$ from Theorem~\ref{thm:local} imply $\omega_{\eta_j}^2\subset\Omega_{(j+1)/10}$ with $\dist(\omega_{\eta_j}^2,\partial\Omega_{(j+1)/10}) \geq \frac{R}{40}$. Moreover, the cut-off functions $\eta_0,\eta_2,\eta_4$ here can be chose to be equal to those in the definition of Theorem~\ref{thm:local}. 
  
 \item \emph{Mapping properties of commutators:} 
Let the commutator of $\A:\widetilde{H}^s(\Omega) \rightarrow H^{-s}(\Omega)$ and an arbitrary cut-off function $\eta$ be defined as the mapping 
 \begin{align}\label{eq:commutator}
  \varphi \mapsto \C_\eta (\varphi) := [\A,\eta](\varphi) := \A(\eta \varphi) - \eta \A(\varphi).
 \end{align}
 The commutator $\C_\eta$  can be seen as a smoothed, localized version of $\A$ since Lemma~\ref{lem:commutator} 
shows the improved mapping property  
 \begin{align}\label{eq:commutatorMapping}
  \C_\eta: \widetilde{H}^{s}(\Omega) \rightarrow H^{1-s}(\Omega).
  \end{align}
We also use commutators of the full-space versions of the fractional Laplacian $(-\Delta)^s$
with cut-off functions $\eta$ 
defined by 
 \begin{align}\label{eq:commutatorFS}
   \varphi \mapsto \widetilde\C_\eta (\varphi) := [(-\Delta)^s,\eta](\varphi)  :=  (-\Delta)^s (\eta \varphi) - \eta (-\Delta)^s \varphi.
 \end{align}
 
\item \emph{Superapproximation:} For $t \in [0,1]$,
  there is a linear operator $\mathcal{J}_h : H^t(\Omega)\rightarrow S^{1,1}_0(\T_h)$ such that for $v_h \in S^{1,1}(\T_h)$
  and cut-off function $\eta_j$ it holds
  \begin{align*}
    \norm{\eta_j v_h - \mathcal{J}_h(\eta_j v_h)}_{H^t(\Omega)} \lesssim  
    h^{2-t} \norm{v_h}_{H^1(\Omega_{(j+1)/10})}.
  \end{align*}
  Here, we gain one power of $h$, since 
  $v_h$ is a discrete function. Lemma~\ref{lem:superapprox} shows that the Scott-Zhang operator~\cite{ScottZhang} has this property.
    Furthermore, due to its local construction, we note that, due to the point (i)
    \begin{align}\label{eq:SZ:supp}
      \supp \mathcal{J}_h(\eta_j v) \subset \omega_{\eta_j}^2 \subset\Omega_{(j+1)/10}.
    \end{align}

\item \emph{Stability of the Galerkin projection:} We define the Galerkin projection $\Pi: \widetilde H^s(\Omega) \rightarrow S_0^{1,1}(\T_h)$ by
\begin{align}\label{eq:Galerkinproj}
 a(\Pi u,v_h) = a(u,v_h) \quad \forall v_h \in S_0^{1,1}(\T_h).
\end{align}
For $t\in [s,1]$, we have
by Lemma~\ref{lem:superapprox} for all $j \in \{0,\dots,9\}$
 \begin{align*}
  \abs{\eta_j \Pi v}_{H^t(\Omega)} \lesssim \abs{v}_{H^t(\Omega)}. 
 \end{align*}
\end{enumerate}
  \begin{lemma}\label{lem:invest}
    Let the assumptions of Theorem~\ref{thm:local} be valid and the sets $\Omega_{j/10}$ and cut-off functions $\eta_j$ be as described above.
    Then, there holds for $j\in\left\{ 0,\ldots, 9 \right\}$ and $t\in[0,1], r\in(-1/2,t]$
    \begin{align}\label{invest:a}
      \| \eta_j v_h \|_{\widetilde H^t(\Omega)} \lesssim h_1^{r-t}\| v_h \|_{H^r(\Omega_{(j+1)/10})}
      \quad\text{ for all } v_h\in S^{1,1}_0(\mathcal{T}_h).
    \end{align}
  \end{lemma}
  \begin{proof}
    To show~\eqref{invest:a} for $r \in [0,t]$, we apply an elementwise inverse estimate using the quasi-uniformity of $\T_{h,1}$,
    \begin{align}\label{invest:eq1}
      \| \eta_j v_h \|_{H^1(\Omega)}^2 \lesssim \!\sum_{T\colon\! T \subset \omega_{\eta_j}}\!\!\! \| v_h \|_{H^1(T)}^2 
      \lesssim \!\sum_{T\colon \!T \subset \omega_{\eta_j}}\!\!\! h_1^{2r-2}\| v_h \|_{H^r(T)}^2 
\lesssim h_1^{2r-2} \| v_h \|_{H^r(\omega_{\eta_j})}^2.\hspace*{-2mm}
    \end{align}
    In particular, $\| \eta_j v_h \|_{H^1(\Omega)} \lesssim h_1^{r-1} \| v_h \|_{H^r(\Omega_{(j+1)/10})}$.
    Furthermore, the support properties of $\eta_j$ and an interpolation arguments gives
    $\| \eta_j v_h \|_{\widetilde H^r(\Omega)} \lesssim \| v_h \|_{H^r(\Omega_{(j+1)/10})}$. Now,~\eqref{invest:a} follows by interpolation between the
  last two estimates. 
  
  Next, we show~\eqref{invest:a} for $-1/2<r<0$.
  The condition $\dist(\omega_{\eta_j}^2,\partial\Omega_{(j+1)/10}) \geq \frac{R}{40}$ together with $h_1/R \leq 1/40$ allows us to choose two cut-off functions
    $\eta'\in\C^\infty_0(\R^d)$ and $\eta_h\in S^{1,1}_0(\mathcal{T}_h)$
    such that $\supp\eta_h\subset\supp\eta'\subset\Omega_{(j+1)/10}$,
    $\eta_h\equiv 1$ on $\omega_{\eta_j}$, $\eta'\equiv 1$ on $\supp\eta_h$, and
    $\norm{\eta_h}_{W^{k,\infty}(\R^d)} + \norm{\eta'}_{W^{k,\infty}(\R^d)} \lesssim R^{-k}$
    for $k \in \{0,1,2\}$. By $\eta_h \equiv 1$ on $\omega_{\eta_j}$, we have
    $\|\eta_jv_h\|_{L^2(\Omega)}\lesssim \|\eta_h v_h \|_{L^2(\Omega)}$.
    Estimate~\eqref{invest:eq1} also gives $\| \eta_j v_h \|_{\widetilde H^1(\Omega)} \lesssim h_1^{-1} \| \eta_hv_h \|_{L^2(\Omega)}$,
    and an interpolation argument shows
    \begin{align}
\label{eq:lemma:invest-10}
      \| \eta_j v_h \|_{\widetilde H^t(\Omega)} \lesssim h_1^{-t} \| \eta_hv_h \|_{L^2(\Omega)}.
    \end{align}
    Denote by $\Pi^{L^2}_2$ the $L^2(\Omega)$-projection onto the piecewise quadratic functions
    \begin{align*}
      S^{2,1}_0(\T_h) &:= \{u \in H^1(\Omega) \,:\, u|_{T} \in \mathcal{P}^2(T) \text{ for all } T \in \T_h\}\cap \widetilde H^1(\Omega).
    \end{align*}
    The boundedness of $\Pi^{L^2}_2$ in $L^2(\Omega)$, an elementwise inverse estimate on $\T_{h,1}$, and quasi-uniformity of $\T_{h,1}$
    show
    \begin{align*}
      \|\eta'\Pi^{L^2}_2 v\|_{L^2(\Omega)} &\lesssim \| v\|_{L^2(\Omega)},\\
      \|\eta'\Pi^{L^2}_2 v\|_{\widetilde H^{1}(\Omega)}^2 &\lesssim
      \sum_{T\in\T_{h,1}}\| \Pi^{L^2}_2 v \|_{H^1(T)}^2
      \lesssim h^{-2}_1 \|\Pi^{L^2}_2 v\|_{L^2(\Omega)}^2 \lesssim h_1^{-2} \| v\|_{L^2(\Omega)}^2.
    \end{align*}
    An interpolation argument then shows $\|\eta'\Pi^{L^2}_2 v\|_{\widetilde H^{-r}(\Omega)} \lesssim h_1^{r} \| v\|_{L^2(\Omega)}$, and hence
    \begin{align}\label{invest:eq2}
      & \| \eta_h v_h\|_{L^2(\Omega)} 
      = \sup_{v\in L^2(\Omega)} \frac{\langle \eta_h v_h, v\rangle_{L^2(\Omega)}}{\|v\|_{L^2(\Omega)}}  
      = \sup_{v\in L^2(\Omega)} \frac{\langle \eta_h v_h, \eta'\Pi^{L^2}_2 v\rangle_{L^2(\Omega)}}{\|v\|_{L^2(\Omega)}}\\
\nonumber 
      &\qquad \leq \| \eta_h v_h\|_{H^{r}(\Omega_{(j+1)/10})}
      \sup_{v\in L^2(\Omega)} \frac{ \|\eta'\Pi^{L^2}_2 v\|_{\widetilde H^{-r}(\Omega)}} {\|v\|_{ L^2(\Omega)}}
      \lesssim h_1^{r}\| v_h\|_{H^{r}(\Omega_{(j+1)/10})}.
    \end{align}
    Inserting (\ref{invest:eq2}) into (\ref{eq:lemma:invest-10}) gives (\ref{invest:a}).
  \end{proof}
\begin{lemma}\label{lem:superapprox}
  Let the assumptions of Theorem~\ref{thm:local} be valid, and let
  $\mathcal{J}_h:\widetilde H^1(\Omega)\rightarrow S^{1,1}_0(\mathcal{T}_h)$ be the Scott-Zhang projection
  from~\cite{ScottZhang}.
  Then, for $-1/2<r\leq s\leq 1$ and $j\in\left\{ 0,\dots,8 \right\}$, there holds
  \begin{align}\label{eq:SZ}
    \norm{\eta_j v_h - \mathcal{J}_h(\eta_j v_h)}_{\widetilde H^s (\Omega)} \lesssim h_1^{1+r-s} \norm{v_h}_{H^r(\Omega_{(j+1)/10})}
    \quad\text{ for all } v_h\in S^{1,1}_0(\mathcal{T}_h), \\
  \label{eq:Galerkinprojc}
    \norm{\eta_j v_h - \Pi(\eta_j v_h)}_{\widetilde H^s (\Omega)} \lesssim h_1^{1+r-s} \norm{v_h}_{H^r(\Omega_{(j+1)/10})}
    \quad\text{ for all } v_h\in S^{1,1}_0(\mathcal{T}_h).
  \end{align}
  Furthermore, for $j,k\in\left\{ 0,\dots,9\right\}$ there holds
  \begin{align}\label{eq:Galerkinprojb}
    \norm{\eta_k\left( 
    \eta_j v_h - \Pi (\eta_j v_h)\right)}_{H^1(\Omega)} &\leq C h_1\norm{v_h}_{H^1(\Omega_{(j+1)/10})}
    \quad\text{ for all } v_h\in S^{1,1}_0(\mathcal{T}_h), \\
 \label{eq:Galerkinproja}
   \norm{\eta_k \left(v - \Pi v\right)}_{H^{1}(\Omega)} &\leq C \norm{v}_{H^{1}(\Omega)}
  \qquad\qquad \quad\text{ for all } v\in \widetilde H^1(\Omega).
 \end{align}
\end{lemma}
\begin{proof}
  Due to~\eqref{eq:SZ:supp} we have $\eta_j v_h - \mathcal{J}_h(\eta_j v_h) \in \widetilde H^1(\Omega)$
  for $j\in\left\{ 0,\dots,9 \right\}$. \linebreak
  \emph{Step 1: }\emph{Proof of~\eqref{eq:SZ} for $r \geq 0$:}
  For $j\in\left\{ 0,\dots,9 \right\}$ and $t\in\left\{ 0,1 \right\}$, we obtain for $0\leq r\leq 1$
  \begin{align*}
    &\norm{\eta_j v_h - \mathcal{J}_h(\eta_j v_h)}_{H^t(\Omega)}^2 \lesssim h_1^{4-2t} 
    \sum_{T\in  \T(\eta_j)}
    \abs{\eta_j v_h}_{H^2(T)}^2 \nonumber \\ 
    &\qquad\qquad\lesssim h_1^{4-2t}
    \sum_{T \in  \T(\eta_j)}
    \|\eta_j\|_{W^{2,\infty}(\mathbb{R}^d)}^2\norm{v_h}_{L^2(T)}^2 + 
    \|\eta_j\|_{W^{1,\infty}(\mathbb{R}^d)}^2\norm{\nabla v_h}_{L^2(T)}^2 \\
    &\qquad\qquad\lesssim h_1^{4-2t} h_{1}^{2r-2}
    \sum_{T\in  \T(\eta_j)}
    \norm{v_h}_{H^r(T)}^2 
    \lesssim 
    h_1^{2-2(t-r)}\norm{v_h}^2_{H^r(\omega_{\eta_j})},
  \end{align*}
  where we additionally used $\abs{v_h}_{H^2(T)} = 0$, the bound $\norm{\eta_j}_{W^{k,\infty}(\R^d)} \lesssim R^{-k}$,
  an elementwise inverse estimate, and the quasi-uniformity of $\T_{h,1}$. 
  Given that $\omega_{\eta_j}\subset\Omega_{(j+1)/10}$, an interpolation argument shows for $0\leq t,r \leq 1$
  \begin{align}\label{lem:superapprox:eq2}
    \norm{\eta_j v_h - \mathcal{J}_h(\eta_j v_h)}_{\widetilde H^t(\Omega)} \lesssim h_1^{1+r-t} \norm{v_h}_{H^r(\Omega_{(j+1)/10})},
  \end{align}
  which implies \eqref{eq:SZ} for $r\geq 0$.
  
  \emph{Step 2: }\emph{Proof of~\eqref{eq:SZ} for $-1/2<r<0$:}
  We will employ the same discrete cut-off function $\eta_h$ used in the proof of Lemma~\ref{lem:invest}.
  Employing Step~1 with $r=0$, we then get with~\eqref{invest:eq2}
  \begin{align*}
    \norm{\eta_j v_h - \mathcal{J}_h(\eta_j v_h)}_{\widetilde H^s(\Omega)} &\lesssim h_1^{1-s}\norm{v_h}_{L^2(\omega_{\eta_j})}
\leq h_1^{1+r-s}\|v_h\|_{H^{r}(\Omega_{(j+1)/10})}.
  \end{align*}
  \emph{Step 3: }\emph{Proof of~\eqref{eq:Galerkinprojc}:} As observed at the beginning,
  $\eta_j v_h - \mathcal{J}_h(\eta_j v_h) \in \widetilde H^1(\Omega)$.
  In particular,~\eqref{eq:Galerkinprojc} is a consequence of~\eqref{eq:SZ} and C\'ea's Lemma.\\
  \emph{Step 4: }\emph{Proof of~\eqref{eq:Galerkinprojb}:}
  We start with the triangle inequality and $\| \eta_k \|_{W^{t,\infty}(\Omega)}\lesssim R^{-k}$,
  \begin{align*}
    \norm{\eta_k\left(\eta_j v_h- \Pi (\eta_j v_h)\right)}_{H^{1}(\Omega)}
    \leq\norm{\eta_j v_h- \mathcal{J}_h(\eta_j v_h)}_{H^{1}(\Omega)} +
    \norm{\eta_k\left(\mathcal{J}_h(\eta_j v_h)- \Pi (\eta_j v_h)\right)}_{H^{1}(\Omega)}.
  \end{align*}
  The first term on the right-hand side can be bounded as stipulated by~\eqref{lem:superapprox:eq2}. 
  For the second term on the right-hand side, we employ the inverse estimate~\eqref{invest:a},
  the triangle inequality, C\'ea's Lemma, and again~\eqref{lem:superapprox:eq2}
  \begin{align*}
    &\norm{\eta_k\left(\mathcal{J}_h(\eta_j v_h)- \Pi (\eta_j v_h)\right)}_{H^{1}(\Omega)}\\
    &\qquad\lesssim h_1^{s-1} \left(\norm{\mathcal{J}_h(\eta_j  v_h)- \eta_j v_h}_{\widetilde H^{s}(\Omega)}
    +\norm{\eta_j  v_h- \Pi (\eta_j v_h)}_{\widetilde H^{s}(\Omega)}\right)\\
    &\qquad\lesssim h_1^{s-1} \norm{\mathcal{J}_h(\eta_j  v_h)- \eta_j v_h}_{\widetilde H^{s}(\Omega)}
    \lesssim h_1\norm{v_h}_{H^{1}(\Omega_{(j+1)/10})}.
  \end{align*}
  \emph{Step 5: }\emph{Proof of~\eqref{eq:Galerkinproja}:}
  We can repeat the arguments of Step~4, replacing superapproximation (\ref{lem:superapprox:eq2})
with the classical approximation properties of the Scott-Zhang projection.
\end{proof}

The following lemma will be essential in our local analysis. It provides mapping properties of the commutator of the fractional Laplacian and a cut-off function 
as well as mapping properties for the commutator of second order.

\begin{lemma}\label{lem:commutator}
 Let $\eta \in C_0^\infty(\R^d)$ and let $\C_\eta$ be the commutator defined in \eqref{eq:commutator} and 
$\widetilde\C_\eta$ be the commutator defined in (\ref{eq:commutatorFS}). 
\begin{enumerate}[(i)]
\item\label{item:lem:commutator-i} The commutator $\widetilde \C_\eta:\widetilde H^{s}(\Omega) \rightarrow H^{1-s}(\R^d)$ is a bounded linear operator. 
\item\label{item:lem:commutator-ii}
  For the commutator $\mathcal{C}_\eta$, we have
 \begin{align*}
  \C_\eta : \widetilde{H}^s(\Omega) \rightarrow H^{1-s}(\Omega) \quad \text{ and, by symmetry (\ref{eq:symmetry-a}),} \quad  \C_\eta : \widetilde H^{s-1}(\Omega) \rightarrow H^{-s}(\Omega).
 \end{align*}
An interpolation argument therefore gives
\begin{align*}
  \C_\eta : \widetilde{H}^{s-1+\theta}(\Omega)  \rightarrow H^{-s+\theta}(\Omega), \qquad \theta \in [0,1].
\end{align*}
\item\label{item:lem:commutator-iii} The commutator of second order $\widetilde{\C}_{\eta,\eta}$ is defined by 
$\widetilde\C_{\eta,\eta} \varphi := \widetilde\C_\eta(\eta \varphi) - \eta \widetilde\C_\eta(\varphi)$. It is a bounded linear operator 
  \begin{align*}
  \widetilde \C_{\eta,\eta} : \widetilde{H}^{s}(\Omega) \rightarrow H^{2-s}(\R^d). 
\end{align*}

\item\label{item:lem:commutator-iv}  For the commutator of second order $\C_{\eta,\eta}$ defined by 
$\C_{\eta,\eta} \varphi := \C_\eta(\eta \varphi) - \eta \C_\eta(\varphi)$, we have 
  \begin{align*}
 &\C_{\eta,\eta} : \widetilde{H}^{s}(\Omega) \rightarrow H^{2-s}(\Omega) \quad \text{ and, by symmetry (\ref{eq:symmetry-b}) and interpolation, } \\ 
& \C_{\eta,\eta} : \widetilde{H}^{s-1/2}(\Omega) \rightarrow H^{3/2-s}(\Omega).
 \end{align*}
\end{enumerate}
\end{lemma}
\begin{proof}
The commutator $\widetilde \C_\eta = [(-\Delta)^s,\eta]$ and the commutator of second order $\widetilde \C_{\eta,\eta}$ have the representations
(for sufficiently smooth $e \in \widetilde H^{s}(\Omega)$) 
\begin{align*}
  \widetilde \C_\eta e (x) &= C(d,s)\text{P.V.} \int_{\R^d} \frac{\eta(x)-\eta(y)}{\abs{x-y}^{d+2s}} e(y) dy, \\
  \widetilde \C_{\eta,\eta} e(x) &= -C(d,s)\int_{\R^d} \frac{(\eta(x)-\eta(y))^2}{\abs{x-y}^{d+2s}} e(y) dy.
\end{align*}
We note that, since $\A$ and $(-\Delta)^s$ are symmetric operators, we have the ``symmetry'' properties 
for $u,v \in \widetilde{H}^s(\Omega)$
\begin{align}
\label{eq:symmetry-a}
\langle \C_\eta u,v\rangle_{L^2(\Omega)} &= - \langle u,\C_\eta v\rangle_{L^2(\Omega)}, 
& \langle \widetilde \C_\eta u,v\rangle_{L^2(\Omega)} & = - \langle u,\widetilde\C_\eta v\rangle_{L^2(\Omega)}, 
\\
\label{eq:symmetry-b}
\langle \C_{\eta,\eta} u,v\rangle_{L^2(\Omega)} &=  \langle u,\C_{\eta,\eta} v\rangle_{L^2(\Omega)}, 
& \langle \widetilde \C_{\eta,\eta} u,v\rangle_{L^2(\Omega)} & =  \langle u,\widetilde\C_{\eta,\eta} v\rangle_{L^2(\Omega)}. 
\end{align}
\emph{Proof of (\ref{item:lem:commutator-i}):} 
Using Taylor expansion, we may write for $n \in \{1,2,3\}$ 
\begin{align*}
  \eta(y)-\eta(x) = \sum_{\alpha \in \N_0^d\colon 1\leq |\alpha| \leq n} \frac{1}{\alpha!} D^\alpha \eta(x) (x-y)^\alpha + R_n(x,y), 
\end{align*}
where the smooth remainder is $\mathcal{O}(|x - y|^{n+1})$. 
Inserting this into the representation of $\widetilde \C_\eta$ shows that we have to analyze convolution operators of the form 
$e \mapsto \kappa_\alpha \star e$, where $\kappa_\alpha(x) = \frac{x^\alpha}{\abs{x}^{d+2s}}$ 
for some $\alpha \in \N_0^d$ with $|\alpha|  \ge 1$. 
Convolution  operators of that form are pseudodifferential operators, \cite{taylor96}, and, in fact, the Fourier transform of 
$\kappa_\alpha$ can be computed explicitly. 
By, e.g., 
\cite[Chap.~{II}, Sec.~3.3]{gelfand64}, we have for the Fourier transform
\begin{align}
\label{eq:landkof-1}
({\mathcal F} \frac{1}{|z|^{d-t}})(\zeta) &= c_{t,d} \abs{\zeta}^{-t}, 
\qquad t \ne  -2m, \quad m \in \N_0,
\\
\label{eq:landkof-2}
({\mathcal F} \ln |z|)(\zeta)  &=  c_{0,d}' |\zeta|^{-d} + c_{0,d}\delta(\zeta) = c_{0,d}' |\zeta|^{-d} + c_{0,d}\mathcal{F}(1)(\zeta),
\end{align}
where $\delta(\cdot)$ is the Dirac delta function. 
Here, $c_{t,d}$, $c_{0,d}$, $c_{0,d}'$ are suitable constants.
The terms $|\zeta|^{-d}$ and $|\zeta|^{-t}$ for $t > 0$ are understood as ``generalized functions'' as described 
in \cite{gelfand64} and expressions involving such terms thus require a regularization as described in \cite[Chap.~{I}, Sec.~{1}]{gelfand64}.
A special role is played by the case $s = 1/2$, for which the Riesz transform arises: 
\begin{align}
\label{eq:riesz-transform}
{\mathcal F} \; \text{P.V.} \frac{z_i}{|z|^{d+1}} & = c' \frac{\zeta_i}{|\zeta|}.
\end{align}

\emph{0.~step:} 
We will ascertain in the following steps of the proof that the Fourier transform $\mathcal{F} \kappa_\alpha$ satisfies for $|\alpha| \leq 3$ 
\begin{align}
\label{eq:symbol-bound-kappa}
| \mathcal{F}\kappa_\alpha(\zeta)|  & \leq C_{\alpha} |\zeta|^{2s - |\alpha|}. 
\end{align}
{}From that we obtain the following mapping property of the convolution operator induced by $\kappa_\alpha$ also denoted by $ \kappa_\alpha: e \mapsto \kappa_\alpha \star e$: For $t \ge 0$ such that 
$t-(2s - |\alpha|) \ge 0$ one has  
\begin{align}
\label{eq:mapping-property-kappa}
\kappa_\alpha: \widetilde{H}^t(\Omega) \rightarrow H^{t-(2s-|\alpha|)}_{loc}(\R^d). 
\end{align}
For $2s -|\alpha| \ge  0$, the assumption (\ref{eq:symbol-bound-kappa}) readily implies the mapping property 
$\kappa_\alpha: H^t(\R^d) \rightarrow H^{t - (2s-|\alpha|)}(\R^d)$ for any $t$. For $2s - |\alpha| < 0$, we observe 
$|\kappa_\alpha(z)| \leq C |z|^{-(d+2s-|\alpha|)}$ so that by the mapping properties of the Riesz potential 
(cf., e.g., \cite[Lemma~{7.12}]{gilbarg-trudinger77a}), we have $\kappa_\alpha:L^2(\Omega) \rightarrow L^2_{loc}(\R^d)$. 
Writing $t - (2s - |\alpha|) = t_1 + \tau$ with $t_1 = \lfloor t - (2s-|\alpha|)\rfloor$ and $\tau = t -(2s - |\alpha|) - t_1 \in [0,1)$,
we have using \cite[Lem.~{3.15}]{mclean00}  for the Slobodecki seminorm $|\cdot|_{H^\tau(\R^d)}$ and any $\widetilde\Omega \subseteq \R^d$
\begin{align*}
\abs{\nabla^{t_1} (\kappa_\alpha \star e)}_{H^\tau(\widetilde{\Omega})} & \leq 
\abs{\nabla^{t_1} (\kappa_\alpha \star e)}_{H^\tau(\R^d)}  \sim 
\norm{|\zeta|^{t_1+\tau} \mathcal{F}(\kappa_\alpha) \mathcal{F}(e)}_{L^2(\R^d)} \\
& \stackrel{(\ref{eq:symbol-bound-kappa})}{\leq} C_\alpha 
\norm{|\zeta|^{t} \mathcal{F}(e)}_{L^2(\R^d)} 
\stackrel{t \ge 0}{\leq } C \norm{e}_{H^{t}(\R^d)}. 
\end{align*} 
Hence, we have arrived at $\|\kappa_\alpha \star e\|_{H^{t-(2s-|\alpha|)}(\widetilde\Omega)} \leq C \|e\|_{\widetilde{H}^{t}(\Omega)}$, 
which is (\ref{eq:mapping-property-kappa}). 

\emph{1.~step:} 
Let  $e\in \widetilde H^s(\Omega)$. Then $\supp e \subset \overline{\Omega}$.
For  $x$ with  large $\abs{x}$, 
the representation of the commutator $\widetilde{\mathcal{C}}_\eta e$ shows that it is a smooth function that decays like $r^{-(d+2s)}$ and its derivative decays like $r^{-(d+1+2s)}$. 
Consequently, in order to show that $\widetilde{\C}_\eta e \in H^{1-s}(\R^d)$, it suffices to assert the 
mapping property $\widetilde \C_\eta:\widetilde H^{s}(\Omega) \rightarrow H^{1-s}_{loc}(\R^d)$.
The same argument also applies to the commutator of second order, where it suffices 
to show   $\widetilde \C_{\eta,\eta} : \widetilde{H}^{s}(\Omega) \rightarrow H^{2-s}_{loc}(\R^d)$.

\emph{2.~step:} Analysis of the remainder $R_n$: 
The remainder induces an operator with kernel $r_n(x,y) = R_n(x,y)/|x - y|^{d+2s}$. Its $x$-derivative  satisfies
\begin{align*}
|\partial_x r_n(x,y)| \leq C |x - y|^{-(d-n+2s)}.
\end{align*}
By the mapping properties of the Riesz potential (cf., e.g., \cite[Lemma~{7.12}]{gilbarg-trudinger77a}), we therefore get the mapping property 
$L^2(\Omega) \rightarrow H^1_{loc}(\R^d)$ provided $2s-n < 0$, i.e., $n \ge  1$ for $s <1/2$ and $n \ge 2$ for $s \ge 1/2$. 

For the second derivative, we similarly have
\begin{align*}
|\partial^2_{x} r_n(x,y)| \leq C |x - y|^{-(d-n+1+2s)},
\end{align*}
and the mapping properties of the Riesz potential imply the mapping property \linebreak
$L^2(\Omega) \rightarrow H^2_{loc}(\R^d)$ provided $1+2s-n < 0$, i.e., $n \ge  2$ for $s <1/2$ and $n \ge 3$ for $s \ge 1/2$. 

\emph{3.~step (estimating $\kappa_\alpha$ for $\abs{\alpha}=1$):} 
For $s \ne 1/2$, we note 
$$
\frac{z_i}{|z|^{d+2s}} = -\frac{1}{d+2s-2} \partial_{z_i} \frac{1}{|z|^{d+2s-2}}.
$$
Using integration by parts in the first order term of the Taylor expansion gives for the  principal value part
upon setting $c_{s,d}:= -(d+2s-2)$
\begin{align*}
c_{s,d} \text{P.V.} \int_{\R^d} \kappa_\alpha(x-y) e(y) dy &= 
\lim_{\varepsilon \rightarrow 0 } \int_{\R^d \backslash B_\varepsilon(x)}\nabla_{y} \frac{1}{|x-y|^{d-(2-2s)}}\cdot e_i e(y) dy \\
&=  \lim_{\varepsilon \rightarrow 0 } \int_{\partial B_\varepsilon(x)}\frac{1}{|x-y|^{d-(2-2s)}} e_i\cdot \nu(y) e(y) ds_y \\
&\qquad\quad -  \int_{\R^d \backslash B_\varepsilon(x)}\frac{1}{|x-y|^{d-(2-2s)}} \partial_{y_i}e(y) dy,
\end{align*}
where $e_i$ is the $i$-th unit vector and $\nu(\cdot)$ denotes the outer normal vector to $B_{\varepsilon}(x)$.
Using Taylor expansion, we write $e(y)= e(x) + \widetilde r_1(x,y)$, where the remainder $\widetilde r_1  = \mathcal{O}(\abs{x-y})$. Then, the boundary integral converges to zero since
\begin{align*}
\abs{\int_{\partial B_\varepsilon(x)}\frac{1}{|x-y|^{d-(2-2s)}} e_i\cdot \nu(y) e(y) ds_y}
&\lesssim\varepsilon^{-d+3-2s} \int_{\partial B_\varepsilon(x)} 1 dy \lesssim  
\varepsilon^{2-2s} \rightarrow 0,
\end{align*}
using that the first term in the expansion vanishes by symmetry.  
We conclude,
\begin{align*}
c_{s,d}\text{P.V.} \int_{\R^d} \kappa_\alpha(x-y) e(y) dy &=  -\int_{\R^d}\frac{1}{|x-y|^{d-(2-2s)}} \partial_{y_i}e(y) dy.
\end{align*}
Therefore, by (\ref{eq:landkof-1}) we get for $|\alpha| = 1$
\begin{align}
\label{eq:fourierorder}
c_{s,d}{\mathcal F}(\kappa_{\alpha}\star e)(\zeta) &= {\mathcal F}\left(\frac{-1}{|z|^{d-(2-2s)}}\star \partial_{z_i} e\right)(\zeta)
 =  -c_{\alpha,d} |\zeta|^{-(2-2s)} \mathbf{i} \zeta_i  {\mathcal F}(e), 
\end{align}
which shows that $|\mathcal{F}\kappa_\alpha (\zeta)| \leq C |\zeta|^{2s-1}$ so that by 
(\ref{eq:mapping-property-kappa}) with $t = s$ and $t  - (2s-1) = 1-s$ we have 
$\kappa_\alpha: \widetilde{H}^s(\Omega) \rightarrow H^{1-s}_{loc}(\R^d)$. 

\emph{4.~step: The case $0 < s <  1/2$:} Selecting $n = 1$, Steps 1 -- 3, show that $\widetilde{\C_\eta} = [(-\Delta)^s,\eta]$ has the mapping property
$\widetilde{H}^s(\Omega) \rightarrow H^{1-s}(\R^d)$.

\emph{5.~step: The case $1/2 < s <  1$:} We use $n = 2$. Again, 
the remainder $R_2$ maps $L^2(\Omega) \rightarrow H^1_{loc}(\R^d)$ by Step~2 and for 
$|\alpha| = 1$, the operator $\kappa_{\alpha}$ is an operator of order $2s-1$ by Step~3. 
The operator $\kappa_{\alpha}$ with $|\alpha| = 2$ is structurally similar to the case $|\alpha| = 1$, since  
we can write
\begin{subequations}\label{eq:derivative2}
\begin{align}
\frac{z_i^2}{|z|^{d+2s}} & = \frac{1}{(d+2s-2)(d+2s-4)}\left(\partial^2_{z_i}\frac{1}{\abs{z}^{d+2s-4}}+\frac{d+2s-4}{\abs{z}^{d+2s-2}}\right),\\
\frac{z_i z_j}{|z|^{d+2s}} & = \frac{1}{(d+2s-2)(d+2s-4)}\left(\partial_{z_i}\partial_{z_j}\frac{1}{\abs{z}^{d+2s-4}}\right).
\end{align}
\end{subequations}
Using again (\ref{eq:landkof-1}) and reasoning as in Step~3 yields 
\begin{equation*}
|\mathcal{F}\kappa_\alpha(\zeta)| \leq C |\zeta|^{-2+2s} 
\end{equation*}
so that by (\ref{eq:mapping-property-kappa}) we have $\kappa_\alpha: \widetilde H^s(\Omega) \rightarrow H^{2-s}_{loc}(\R^d) \subset H^{1-s}_{loc}(\R^d)$. 

\emph{6.~step: The case $s = 1/2$:} 
We use $n = 2$.  For $|\alpha| = 1$, the kernel $\kappa_\alpha$ is the Riesz transform that is, by 
the representation (\ref{eq:riesz-transform}), an operator of order $0 = 2s-1$. For $|\alpha| = 2$, \eqref{eq:derivative2} 
can be used for $d \not\in \{1,3\}$ showing that $\kappa_\alpha$ induces an operator of order $2s-1$.
In the case $d=1$, the kernel $\kappa_\alpha$ with $|\alpha| = 2$ is bounded by 1. For $d=3$
and $|\alpha| = 2$, we use 
\begin{align*}
2\frac{z_i^2}{|z|^{d+1}} = -\partial_{z_i}^2 \ln |z| + \frac{1}{\abs{z}^2}.
\end{align*}
By (\ref{eq:landkof-2}) and  (\ref{eq:landkof-1}) and integration by parts, we have as in \eqref{eq:fourierorder}
\begin{align*}
2{\mathcal F}(\kappa_{\alpha}\star e)(\zeta) &= -{\mathcal F}\left(\ln |z|\star \partial^2_{z_i} e\right)(\zeta) + {\mathcal F}\left(\abs{z}^{-2}\star e\right)(\zeta) \\ 
 &=  c_{\alpha,d}' |\zeta|^{-3} {\mathcal F}(\partial^2_{z_i} e) + c_{\alpha,d} {\mathcal F(1)}{\mathcal F}(\partial^2_{z_i} e) + \tilde c_{\alpha,d} |\zeta|^{-1} {\mathcal F}(e) \\
 &=  \left( - c_{\alpha,d}^\prime \zeta_i^2 |\zeta|^{-3} + \tilde c_{\alpha,d}|\zeta|^{-1} \right) {\mathcal F}(e),
\end{align*}
which implies by (\ref{eq:mapping-property-kappa}) that $\kappa_{|\alpha|}$ induces an operator of order $-1 = 2s-2$. 

Altogether, this gives the boundedness of $\widetilde \C_\eta :\widetilde H^{s}(\Omega) \rightarrow H^{1-s}(\R^d)$ for all $s \in (0,1)$. 

\emph{Proof of (\ref{item:lem:commutator-iii}):} We use $n=3$. Taylor expansion and
the representation of $\widetilde{\mathcal{C}}_{\eta,\eta}$ shows that, due to $(\eta(x)-\eta(y))^2$  in the numerator, the leading order term produces $\kappa_\alpha$ 
with $\abs{\alpha} = 2$ and leads to an operator of order $2s-2$ as in Step~5 in the proof of (\ref{item:lem:commutator-i}). 
The terms with $\abs{\alpha} = 3$ are structurally similar to those for $\abs{\alpha} = 2$.
We have
\begin{align*}
\frac{z_i^3}{|z|^{d+2s}} & =-\frac{1}{(d+2s-2)(d+2s-4)(d+2s-6)}\left(\partial^3_{z_i}\frac{1}{\abs{z}^{d+2s-6}}-\partial_{z_i}\frac{3(d+2s-6)}{\abs{z}^{d+2s-4}}\right),
\end{align*}
and similar expressions hold for the mixed derivatives and the logarithm (for the case $s=1/2$). Therefore,
 we again can use integration by parts and \eqref{eq:landkof-1}, \eqref{eq:landkof-2} to obtain that  $\kappa_{\alpha}$ is an operator of order $2s-3$ for $\abs{\alpha} =3$.

The remainder $R_3$ maps $L^2(\Omega) \rightarrow H^2_{loc}(\R^d)$ by Step~2  in the proof of (\ref{item:lem:commutator-i}) and together this shows that $\widetilde{\mathcal{C}}_{\eta,\eta}$ is an operator of order $2s-2$.

\emph{Proof of (\ref{item:lem:commutator-ii}) and (\ref{item:lem:commutator-iv}):} As the operators $\widetilde \C_\eta$ and $\widetilde{\mathcal{C}}_{\eta,\eta}$ are extensions
of the operators $\C_\eta$ and $\mathcal{C}_{\eta,\eta}$ respectively, the boundedness $\C_\eta : \widetilde H^s(\Omega) \rightarrow H^{1-s}(\Omega)$ follows from (\ref{item:lem:commutator-i}) and the boundedness  $\C_{\eta,\eta} : \widetilde H^s(\Omega) \rightarrow H^{2-s}(\Omega)$ follows from (\ref{item:lem:commutator-iii}). 
The symmetry  
property (\ref{eq:symmetry-a})
of $\C_\eta$ then immediately implies
$  \C_\eta : \widetilde{H}^{s-1}(\Omega) \rightarrow H^{-s}(\Omega)$ as a bounded operator. Finally, both these mapping properties imply 
$\C_\eta :\widetilde{H}^{s-1+\theta}(\Omega)  \rightarrow H^{-s+\theta}(\Omega)$ for $\theta \in [0,1]$ by interpolation. The same argument gives the additional mapping property of the commutator of second order.
\end{proof}

We start with the proof of the first statement in Theorem~\ref{thm:local}, the local error estimate in the energy norm.

\begin{proof}[Proof of Theorem~\ref{thm:local}, (\ref{item:thm:local-i})]
 We write using the Galerkin projection $\Pi$ from \eqref{eq:Galerkinproj}, 
 the symmetry of $\A $, and the definition of $\C_\eta$ from \eqref{eq:commutator}
\begin{align*}
\norm{\eta_0 e}_{\widetilde H^s(\Omega)}^2 \! &\lesssim \! \skp{\A(\eta_0 e),\eta_0 e}_{L^2(\Omega)} = 
\skp{\A(\eta_0 e\!-\!\Pi(\eta_0 e)),\eta_0 e}_{L^2(\Omega)} \! + \!  \skp{\A(\Pi(\eta_0 e)),\eta_0 e}_{L^2(\Omega)} \\
&=\! \skp{\A(\eta_0 e-\Pi(\eta_0 e)),\eta_0 e}_{L^2(\Omega)} +  \skp{\Pi(\eta_0 e), \C_{\eta_0} e}_{L^2(\Omega)} +
\skp{\Pi(\eta_0 e), \eta_0\A e}_{L^2(\Omega)} \\
&=: 
{\large\textcircled{\normalsize \text{I}}} + {\large\textcircled{\normalsize \text{II}}}+ {\large\textcircled{\normalsize \text{III}}}.
\end{align*}
The mapping properties of $\A$,
  the stability of the Galerkin projection, the bound~\eqref{eq:Galerkinprojc} with $r=s-1/2$ from Lemma \ref{lem:superapprox},
  and $\eta_2|_{\Omega_{1/10}}\equiv 1$ lead to 
\begin{align*}
  \abs{{\large\textcircled{\normalsize \text{I}}}} &=  \abs{\skp{\eta_0 e-\Pi(\eta_0 e),\A(\eta_0 e)}_{L^2(\Omega)} }
  \\ &\leq \norm{\A(\eta_0 e)}_{H^{-s}(\Omega)} 
  \left(\norm{\eta_0 u - \Pi(\eta_0 u)}_{\widetilde H^s(\Omega)} +  \norm{\eta_0 u_h - \Pi(\eta_0 u_h)}_{\widetilde{H}^s(\Omega)} \right) \\
  &\lesssim \norm{\eta_0 e}_{\widetilde H^s(\Omega)}\left(\norm{\eta_0 u}_{\widetilde H^s(\Omega)}
  + h_1^{1/2} \norm{u_h}_{H^{s-1/2}(\Omega_{3/10})}\right)\\
  &\lesssim \norm{\eta_0 e}_{\widetilde H^s(\Omega)}\left(\norm{\eta_0 u}_{\widetilde H^s(\Omega)}+h^{1/2}_1\norm{\eta_2 u_h}_{H^{s-1/2}(\Omega)}\right) \\
  &\lesssim \norm{\eta_0 e}_{\widetilde H^s(\Omega)}\left((1+h^{1/2}_1)\norm{\eta_2 u}_{\widetilde H^s(\Omega)}+h^{1/2}_1\norm{e}_{H^{s-1/2}(\Omega)}\right).
\end{align*}
Next, note 
  $\norm{e}_{\widetilde H^{s-1}(\Omega)}\lesssim \norm{e}_{H^{s-1/2}(\Omega)}$.
  For $1/2\leq s$ this is clear since $H^{s-1/2}(\Omega)\subset L^2(\Omega)\subset \widetilde H^{s-1}(\Omega)$.
  For $s<1/2$, we recall that $\widetilde H^{1/2-s}(\Omega)=H^{1/2-s}(\Omega)$ with equivalent norms, and hence
  $H^{s-1/2}(\Omega)=\widetilde H^{s-1/2}(\Omega) \subset \widetilde H^{s-1}(\Omega)$.
  Hence, the mapping properties of the commutator from Lemma~\ref{lem:commutator} and the
  stability of the Galerkin projection imply
\begin{align*}
  \abs{{\large\textcircled{\normalsize \text{II}}}} \! = \! \abs{\skp{\Pi(\eta_0 e), \C_{\eta_0} e}_{L^2(\Omega)}} \!\lesssim \!
  \norm{\Pi(\eta_0 e)}_{\widetilde H^s(\Omega)} \norm{\C_{\eta_0} e}_{H^{-s}(\Omega)}\lesssim  \norm{\eta_0 e}_{\widetilde H^s(\Omega)} \norm{e}_{H^{s-1/2}(\Omega)}.
\end{align*}
It remains to estimate $\large\textcircled{\normalsize \text{III}}$.
With the Galerkin orthogonality,~\eqref{eq:SZ:supp}, as well as $\eta_1\equiv 1$ on $\Omega_{1/10}$, we obtain
\begin{align*}
\abs{\large\textcircled{\normalsize \text{III}}}&=\abs{\skp{\A e,\eta_0 \Pi(\eta_0 e)}_{L^2(\Omega)}}
=\abs{\skp{\eta_1\A e,\eta_0 \Pi(\eta_0 e)-\mathcal{J}_h(\eta_0 \Pi(\eta_0 e))}_{L^2(\Omega)}}.
\end{align*}
Then, the bound~\eqref{eq:SZ} from Lemma \ref{lem:superapprox}, the estimate
$\| \cdot \|_{H^s(\Omega_{1/10})}\lesssim \| \cdot \|_{\widetilde H^s(\Omega)}$, Lemma~\ref{lem:commutator}, the stability of the Galerkin projection $\Pi$, 
the inverse estimate from Lemma~\ref{lem:invest},
the fact $\eta_{2}|_{\Omega_{2/10}}\equiv 1$,
and $\norm{e}_{\widetilde H^{s-1}(\Omega)}\lesssim \norm{e}_{H^{s-1/2}(\Omega)}$ as above lead to
\begin{align*}
\abs{\large\textcircled{\normalsize \text{III}}}
&\lesssim \norm{\eta_1 \A e}_{H^{-s}(\Omega)}\norm{\eta_0 \Pi(\eta_0 e)-\mathcal{J}_h(\eta_0 \Pi(\eta_0 e))}_{\widetilde H^{s}(\Omega)}
\\
&\stackrel{(\ref{eq:SZ})}{ \lesssim } \norm{\A(\eta_1 e) - \mathcal{C}_{\eta_1}e}_{H^{-s}(\Omega)}h_1\norm{\Pi(\eta_0 e)}_{\widetilde H^{s}(\Omega)} \\
&
\stackrel{\text{Lem.~\ref{lem:commutator}}}{
\lesssim} \left(\norm{\eta_1 e}_{\widetilde H^{s}(\Omega)}+\norm{e}_{H^{s-1/2}(\Omega)}\right) h_1 \norm{\eta_0 e}_{\widetilde H^{s}(\Omega)} \\
&\stackrel{(\ref{invest:a})}{\lesssim }
\left(h_1\norm{\eta_1 u}_{\widetilde H^{s}(\Omega)}+h^{1/2}_1\norm{\eta_2 u_h}_{H^{s-1/2}(\Omega)}+h_1\norm{e}_{H^{s-1/2}(\Omega)}\right) 
\norm{\eta_0 e}_{\widetilde H^{s}(\Omega)} \\
&\lesssim 
(h^{1/2}_1+h_1)
\left(\norm{\eta_2 u}_{\widetilde H^{s}(\Omega)}+\norm{e}_{H^{s-1/2}(\Omega)}\right)
\norm{\eta_0 e}_{\widetilde H^{s}(\Omega)}.
\end{align*}
Putting the estimates of the three terms together and using $h_1 \lesssim 1$, we obtain
\begin{align*}
 \norm{\eta_0 e}_{\widetilde H^s(\Omega)}^2 \lesssim 
 \norm{\eta_0 e}_{\widetilde H^s(\Omega)}\left(\norm{\eta_2u}_{\widetilde H^s(\Omega)}+\norm{e}_{H^{s-1/2}(\Omega)}\right).
\end{align*}
Applying this estimate to $u - v_h$ for arbitrary $v_h \in S^{1,1}_0(\T_h)$ instead of $u$ and noting that the corresponding Galerkin error 
is $e = (u-v_h) +(v_h-u_h)$  leads to the desired estimate 
in the energy norm.
\end{proof}

\begin{remark}\label{rem:graded}
  Lemmas~\ref{lem:invest},~\ref{lem:superapprox}, and in turn Theorem~\ref{thm:local}(\ref{item:thm:local-i}), require meshes $\T_h$ that are
  quasi-uniform on $\Omega_1$. It is possible to extend relax this condition. 
  Then the factor $h^{1/2}_1$ in the proof of Theorem~\ref{thm:local} has to be replaced by $h_1 h_{1,\min}^{-1/2}$.  
Therefore, if 
\begin{align*}
  \frac{h_1}{h_{1,\min}^{1/2} }\leq C
\end{align*}
with a constant $C>0$ independent of the local mesh sizes, the previous arguments give the sharp local error estimate
\begin{align}
 \norm{\eta_0(u-u_h)}_{\widetilde H^s(\Omega)} \lesssim  \operatorname*{inf}_{v_h \in S^{1,1}_0(\T)} \norm{\eta_4(u- v_h)}_{\widetilde H^s(\Omega)} + 
 \norm{u-u_h}_{H^{s-1/2}(\Omega)}. 
\end{align}
\eremk
\end{remark}

In the following, we focus on the case of local estimates in 
  the stronger $H^1$-seminorm as stated in Theorem~\ref{thm:local}.
In the proof, we exploit additional interior regularity provided by the following lemma.
\begin{lemma}\label{lem:intreg}
Let $\widehat\Omega \subset\subset \Omega$ be open and $\eta$ be a cut-off function with $\supp \eta \subset \widehat\Omega$.
Assume $f \in H^t(\widehat{\Omega})\cap H^{-s}(\Omega)$ for some $-s\leq t\leq 1-s$ and let $u$ solve \eqref{eq:modelproblem}. Then,
$ \eta u \in H^{2s+t}(\R^d)$ and 
\begin{align*}
\norm{\eta u}_{H^{2s+t}(\R^d)} \lesssim \norm{u}_{\widetilde H^{s}(\Omega)} +   \norm{\eta f}_{H^{t}(\R^d)}.
\end{align*}
\end{lemma}
\begin{proof}
By definition of the commutator $\widetilde{\mathcal{C}}_\eta$, the product $\eta u$ solves the equation
\begin{align}
\label{eq:vollraumgleichung}
(-\Delta)^s(\eta u) + \eta u = \eta (-\Delta)^s u + \widetilde{\mathcal{C}}_{\eta} u + \eta u = \eta f + \widetilde{\mathcal{C}}_{\eta} u + \eta u =: \widetilde{f}.
\end{align}
Since $\eta f \in H^t(\R^d)$ and $\widetilde{\mathcal{C}}_{\eta} u \in H^{1-s}(\R^d)$ by Lemma~\ref{lem:commutator} 
and $\eta u \in H^s(\R^d)$,
we have $\widetilde{f} \in H^{\min\{t,1-s,s\}}(\R^d)$. Applying the Fourier transformation to (\ref{eq:vollraumgleichung}) as in the proof of Lemma~\ref{lem:commutator}
noting that all objects live in the full-space $\R^d$, gives 
\begin{align*}
 (1+\abs{\zeta}^{2s}) \mathcal{F}(\eta u) = \mathcal{F}(\widetilde{f}),
\end{align*}
and the Fourier definition of Sobolev norms implies
$\eta u \in H^{2s+\min\{t,1-s,s\}}(\R^d)$. Bootstrapping this argument until the minimum in the exponent is given by 
$t$ then shows the claimed local regularity. 
The norm estimate follows directly from the above equation and the Fourier definition of Sobolev norms and 
the mapping properties of the commutator $\widetilde{\mathcal{C}}_\eta$ from Lemma~\ref{lem:commutator}.
\end{proof}

We will repeatedly employ the $L^2$-orthogonal projection $\Pi^{L^2}:L^2(\Omega) \rightarrow S^{1,1}_0(\T_h)$ defined by
\begin{align}\label{def:L2projection}
\skp{\phi-\Pi^{L^2}\phi,\xi_h}_{L^2(\Omega)} = 0 \qquad \forall \xi_h \in S^{1,1}_0(\T_h).
\end{align}
There hold the following global stability and approximation estimates.
\begin{lemma}\label{lem:L2}
  Let $s\in[0,1]$ and $\T_h$ be a shape regular mesh of size $h$.
Then, the following approximation estimates in negative norms hold: 
  \begin{align*}
    \norm{\phi-\Pi^{L^2}\phi}_{\widetilde H^{s-1}(\Omega)} &\leq C h \norm{\phi}_{\widetilde H^s(\Omega)} \quad\text{ for } s\in (1/2,1],\\
    \norm{\phi-\Pi^{L^2}\phi}_{\widetilde H^{s-1}(\Omega)} &\leq C h^{3/2-2\varepsilon} \norm{\phi}_{\widetilde H^{1/2+s-\varepsilon}(\Omega)} \quad\text{ for } s\in[0,1/2],\, \varepsilon \in (0,1/2],
  \end{align*}
where the constant $C>0$ depends only on $\Omega$, $d$, $s$, $\varepsilon,$ and the $\gamma$-shape regularity of $\T_h$.
\end{lemma}
\begin{proof}
  The approximation estimate for $s\in(1/2,1]$ follows by a simple duality argument and the fact
  that $H^{1-s}(\Omega) = \widetilde H^{1-s}(\Omega)$ with equivalent norms. 
  In the case $s\in[0,1/2]$, the definition of the $\widetilde H^{s-1}(\Omega)$-norm, the orthogonality 
  and standard first order approximation results of the 
  $L^2$-orthogonal projection directly give the stated estimate.
\end{proof}
The $L^2$-orthogonal projection has the advantage that it localizes very well, which is observed, e.g., in 
\cite{wahlbin91}. The following lemma summarizes the local stability and approximation properties of the $L^2$-projection
used in the proof of our main result.

\begin{lemma}\label{lem:localL2}
  Let the assumptions of Theorem~\ref{thm:local} be valid. 
  Let $\Pi^{L^2}$ be the $L^2$-projection defined in \eqref{def:L2projection} and $D_0\subset D_1 \subset \Omega_1$
  be nested open sets with $r:=\dist(D_0,\partial D_1)\geq 4h_1$.
  Then, for $\phi \in H^1(D_1)\cap L^2(\Omega)$, we have local stability
 \begin{align}\label{eq:L2projectionH1}
   \norm{\Pi^{L^2}\phi}_{H^1(D_0)} \leq C\left(\norm{\phi}_{H^1(D_1)} + e^{-cr/h_1}\norm{\phi}_{L^2(\Omega)}\right)
 \end{align}
 and approximation properties
  \begin{align}\label{eq:L2projectionApprox}
    \norm{\phi-\Pi^{L^2}\phi}_{H^t(D_0)} \leq C\left( h^{1-t}_{1}\norm{\phi}_{H^1(D_1)} + e^{-cr/h_1}\norm{\phi}_{L^2(\Omega)}\right)
 \end{align}
for $t \in [0,1]$, where the constants $c$, $C>0$ depend only on $\Omega$, $d$, $D_0$, $D_1$, $c_0$, $t$, and the $\gamma$-shape regularity of $\T_h$.
\end{lemma}
\begin{proof}
Note that~\eqref{eq:L2projectionH1} follows from~\eqref{eq:L2projectionApprox} with $t=1$ and the triangle inequality.
Hence, we focus on the latter estimate.

\emph{Step 1:}
Let  $D_{1/2}$ be a set satisfying $D_0\subset D_{1/2}\subset D_1$ with $\dist(D_0,\partial D_{1/2}) = \dist(D_{1/2},\partial D_{1}) \geq r/2$.
We use \cite[Lem.~7.1]{wahlbin91}, which states that discrete functions $\phi_h \in S^{1,1}(\T_h)$ satisfying the orthogonality
$\skp{\phi_h,\xi_h}_{L^2(\Omega)} = 0$ for all $\xi_h \in S^{1,1}_0(\T_h) $, $\supp \xi_h \subset D_{1/2}$, are exponentially
small locally, i.e., 
\begin{align}\label{eq:L2expsmall}
  \norm{\phi_h}_{L^2(D_0)} \lesssim e^{-c_1r/h_1} \norm{\phi_h}_{L^2(D_{1/2})};
\end{align}
here, one uses that the mesh on $\Omega_1$ is quasi-uniform.
We employ a cut-off function $\eta$ with $\eta \equiv 1$ on $D_{1/2}$, $\supp \eta \subset D_1$,
and $\| \eta \|_{W^{k,\infty}(\R^d)}\lesssim r^{-k}$.
By the definition of the $L^2$-projection, we compute
\begin{align*}
\skp{\Pi^{L^2}\phi,\xi_h}_{L^2(\Omega)} &= \skp{\phi,\xi_h}_{L^2(\Omega)} = \skp{\eta\phi,\xi_h}_{L^2(\Omega)} 
\\&=  
\skp{\Pi^{L^2}(\eta\phi),\xi_h}_{L^2(\Omega)} \forall  \xi_h \in S^{1,1}_0(\T_h), \supp \xi_h \subset D_{1/2}.
\end{align*}
Therefore, we may use \eqref{eq:L2expsmall} with $\phi_h = \Pi^{L^2}\phi-\Pi^{L^2}(\eta\phi)$ and conclude 
\begin{align}
\label{eq:lem:localL2-10}
  \norm{\Pi^{L^2}\phi-\Pi^{L^2}(\eta\phi)}_{L^2(D_0)} \lesssim e^{-c_1r/h_1} \norm{\Pi^{L^2}\phi-\Pi^{L^2}(\eta\phi)}_{L^2(D_{1/2})}.
\end{align}
\emph{Step 2:}Let $v \in H^1(D_1)$. The support properties of $\eta$ imply 
$\|( \operatorname{I}- \Pi^{L^2}) (\eta v)\|_{L^2(D_{1/2})} \leq \|\eta (\operatorname{I}- \Pi^{L^2}) (\eta v)\|_{L^2(\Omega)}$
and for arbitrary $w \in L^2(\Omega)$, using the local approximation properties of the Scott-Zhang operator $\mathcal{J}_h$, we calculate
\begin{align*}
& \left| \langle \eta (\operatorname{I} - \Pi^{L^2})(\eta v),w\rangle_{L^2(\Omega)} \right|  = 
\left| \langle (\operatorname{I} - \Pi^{L^2}) (\eta v), (\operatorname{I}-  \Pi^{L^2}) (\eta w)\rangle_{L^2(\Omega)} \right| \\
& \qquad = \left| \langle (\operatorname{I} - {\mathcal J}_h) (\eta v), (\operatorname{I}-  \Pi^{L^2}) (\eta w)\rangle_{L^2(\Omega)} \right| 
\lesssim h_1 \| v\|_{H^1(D_1)} \|w\|_{L^2(\Omega)}. 
\end{align*}
This implies
\begin{equation}\label{eq:lem:localL2-15}
  \|( \operatorname{I}- \Pi^{L^2}) (\eta v)\|_{L^2(D_0)} \leq
  \|( \operatorname{I} - \Pi^{L^2}) (\eta v)\|_{L^2(D_{1/2})} \lesssim h_1\|v\|_{H^1(D_1)}.
\end{equation}
Using the quasi-uniformity of the mesh ${\mathcal T}_{h,1}$, we apply an elementwise inverse estimate,
the triangle inequality,~\eqref{eq:lem:localL2-15}, the local approximation properties of the Scott-Zhang operator, and the properties of $\eta$,
to obtain
\begin{align*}
  \|\Pi^{L^2} (\eta v) - \mathcal{J}_h(\eta v)\|_{H^1(D_0)}
  &\lesssim h_1^{-1} \|\Pi^{L^2} (\eta v) - \eta v\|_{L^2(D_{1/2})}
    + h_1^{-1} \|\eta v - \mathcal{J}_h(\eta v)\|_{L^2(D_{1/2})}\\
    &\lesssim \| \eta v \|_{H^1(D_1)} \lesssim \| v \|_{H^1(D_1)}.
\end{align*}
Together with the triangle inequality, local stability properties of the Scott-Zhang operator, and the properties of $\eta$, this gives
\begin{align*}
  \|\eta v - \Pi^{L^2} (\eta v)\|_{H^1(D_0)} & \lesssim \|\Pi^{L^2} (\eta v) - \mathcal{J}_h(\eta v)\|_{H^1(D_0)} + \| v \|_{H^1(D_1)}
  \lesssim \| v \|_{H^1(D_1)}.
\end{align*}
Interpolating this with~\eqref{eq:lem:localL2-15} shows for $t\in[0,1]$
\begin{equation}\label{eq:lem:localL2-20}
  \|( 1- \Pi^{L^2}) (\eta v)\|_{H^t(D_0)} \lesssim h_1^{1-t} \|v\|_{H^1(D_1)}. 
\end{equation}
\emph{Step 3:} We use $\eta=1$ on $D_0$ and $t\leq 1$ to write
\begin{align*}
 \norm{\phi - \Pi^{L^2}\phi}_{H^t(D_0)} &\leq  \norm{\eta\phi - \Pi^{L^2}(\eta\phi)}_{H^t(D_0)} + \norm{\Pi^{L^2}\phi- \Pi^{L^2}(\eta\phi)}_{H^1(D_0)}. 
\end{align*}
The first term on the right-hand side is bounded as desired by~\eqref{eq:lem:localL2-20}, while for the second term we use an inverse estimate
and~\eqref{eq:lem:localL2-10},
\begin{align*}
 \norm{\Pi^{L^2}\phi- \Pi^{L^2}(\eta\phi)}_{H^1(D_0)} 
 &\lesssim h^{-1}_1\norm{\Pi^{L^2}\phi- \Pi^{L^2}(\eta\phi)}_{L^2(D_0)}\\
 &\lesssim h^{-1}_1 e^{-c_1r/h_1}\norm{\Pi^{L^2}\phi- \Pi^{L^2}(\eta\phi)}_{L^2(D_{1/2})}
\\ &
\lesssim e^{-cr/h_1}\norm{\phi}_{L^2(\Omega)},
\end{align*} 
where we used 
$h_1^{-1}e^{-c_1r/h_1}\lesssim e^{-cr/h_1}$.
\end{proof}

\begin{proof}[Proof of Theorem~\ref{thm:local}(\ref{item:thm:local-ii})]
With the triangle inequality and the $L^2$-orthogonal projection $\Pi^{L^2}$, we divide the error into three contributions
\begin{align*}
  \norm{e}_{H^1(\Omega_0)} &\leq \norm{\eta_0\eta_4 e}_{H^1(\Omega)} \\ &\leq \norm{\eta_0(\eta_4 e \!-\! \Pi(\eta_4 e))}_{H^1(\Omega)}\! + \!\norm{\eta_0\Pi^{L^2}\phi}_{H^1(\Omega)} \!\! + \!
  \norm{\eta_0(\Pi(\eta_4 e) \!-\! \Pi^{L^2}\phi)}_{H^1(\Omega)}  \\ &=: 
 {\large\textcircled{\normalsize \text{1}}}+{\large\textcircled{\normalsize \text{2}}}+{\large\textcircled{\normalsize \text{3}}} ,
\end{align*}
where 
\begin{align}
\label{eq:phi} 
\phi := \A^{-1}(\eta_4 \C_{\eta_4} e). 
\end{align}
The three terms 
$\textcircled{\text{1}}$, 
$\textcircled{\text{2}}$, 
$\textcircled{\text{3}}$, 
are estimated in the following in turn. 

{\bf Estimate of {\large\textcircled{\normalsize \text{1}}}:}
We use the bounds~\eqref{eq:Galerkinprojb} and~\eqref{eq:Galerkinproja} from Lemma~\ref{lem:superapprox} with $k=0$, $j=4$
and the triangle inequality to obtain
\begin{align}
\label{eq:estimate-circled-1}
  \textcircled{\normalsize\text{1}} & = \norm{\eta_0(\eta_4 e \!-\! \Pi(\eta_4 e))}_{H^1(\Omega)}
  \lesssim  \norm{\eta_4 u}_{H^1(\Omega)} + h_{1}\norm{u_h}_{H^1(\Omega_{5/10})} \\
\nonumber 
&\lesssim (1+h_{1})\norm{u}_{H^1(\Omega_{5/10})} + h_{1}\norm{e}_{H^1(\Omega_{5/10})}.
\end{align}

{\bf Estimate of {\large\textcircled{\normalsize \text{2}}}:}
By Lemma~\ref{lem:commutator} we have $\eta_4\mathcal{C}_{\eta_4} e \in H^{1-s}(\Omega)$, and Lemma~\ref{lem:intreg} then implies
$\eta_4\phi = \eta_4 \A^{-1}(\eta_4 \mathcal{C}_{\eta_4} e) \in H^{1+s}(\R^d)\subset H^1(\R^d)$. 
The local $H^1$-stability of the $L^2$-orthogonal projection from \eqref{eq:L2projectionH1}, applied with $D_0 = \Omega_{1/10}$ and $D_1=\Omega_{2/10}$  
that satisfy $\dist(D_0,\partial D_1) \ge  R/10 \geq 4h_1$, then implies
\begin{align*}
  \norm{\eta_0\Pi^{L^2}\phi}_{H^1(\Omega)} \lesssim \norm{\phi}_{H^1(\Omega_{2/10})} + e^{-c/h_1} \norm{\phi}_{L^2(\Omega)} \lesssim 
  \norm{\eta_4 \phi}_{H^1(\R^d)} + e^{-c/h_1} \norm{\phi}_{\widetilde H^s(\Omega)},
\end{align*}
where the constant $c$ additionally depends on $R$ but not on $h_{1}$.
Lemma~\ref{lem:intreg} yields
\begin{align*}
  \norm{\eta_4 \phi}_{H^1(\R^d)}= \norm{\eta_4\A^{-1}(\eta_4 \C_{\eta_4} e)}_{H^1(\Omega)}&
  \lesssim  \norm{\phi}_{\widetilde H^{s}(\Omega)} + \norm{\eta_4^2 \C_{\eta_4} e}_{\widetilde H^{1-2s}(\Omega)} \\
 & \lesssim \norm{\phi}_{\widetilde H^{s}(\Omega)}+ \norm{\eta_4\C_{\eta_4} e}_{H^{1-2s}(\Omega)}.
\end{align*}
As $1-2s<3/2-s$, the mapping properties of $\C_{\eta_4}$, $\C_{\eta_4,\eta_4}$ from Lemma~\ref{lem:commutator} 
lead to
\begin{align}\label{eq:estphih}
 \norm{\eta_4 \phi}_{H^1(\R^d)} &\lesssim \norm{\phi}_{\widetilde H^{s}(\Omega)}+\norm{\C_{\eta_4}(\eta_4 e)}_{H^{1-2s}(\Omega)} +\norm{\C_{\eta_4,\eta_4} e}_{H^{1-2s}(\Omega)} \nonumber \\
  &\lesssim \norm{\phi}_{\widetilde H^{s}(\Omega)}+\norm{\C_{\eta_4}(\eta_4 e)}_{H^{1-2s}(\Omega)} +\norm{\C_{\eta_4,\eta_4} e}_{H^{3/2-s}(\Omega)} \nonumber \\
  &\lesssim \norm{\phi}_{\widetilde H^{s}(\Omega)} + \norm{\eta_4 e}_{L^2(\Omega)} + \norm{e}_{H^{s-1/2}(\Omega)}.
\end{align}
It remains to bound $\norm{\phi}_{\widetilde H^{s}(\Omega)}$. 
Assumption~\ref{ass:shift} applied with 
right-hand side $\eta_4 \mathcal{C}_{\eta_4} e \in H^{1-s}(\Omega)$ gives
$\phi \in \widetilde H^{1/2+s-\varepsilon}(\Omega)$. Together with Lemma~\ref{lem:commutator}, this implies
\begin{align}\label{eq:estphihglob}
 \norm{\phi}_{\widetilde H^s(\Omega)} \leq \norm{\phi}_{\widetilde H^{1/2+s-\varepsilon}(\Omega)} &=
\norm{\A^{-1}(\eta_4\mathcal{C}_{\eta_4} e)}_{\widetilde H^{1/2+s-\varepsilon}(\Omega)} \\
\nonumber 
&\lesssim  \norm{\eta_4\mathcal{C}_{\eta_4} e}_{\widetilde H^{1/2-s-\varepsilon}(\Omega)} 
\lesssim \norm{e}_{H^{s-1/2}(\Omega)}.
\end{align}
We conclude 
\begin{equation}
\label{eq:estimate-circled-2}
\textcircled{\normalsize{\text{2}}} \lesssim \|\eta_4 e\|_{L^2(\Omega)} + \|e\|_{H^{s-1/2}(\Omega)}. 
\end{equation}

{\bf Estimate of {\large\textcircled{\normalsize \text{3}}}:}
Define the discrete function
\begin{equation}
\label{eq:psih}
\psi_h := \Pi(\eta_4 e)- \Pi^{L^2}\phi
\end{equation}
and decompose $\psi_h = v_h + w_h$,
where $v_h,$ $w_h \in S^{1,1}_0(\T_h)$ solve
\begin{align}
\label{eqvh}
 \skp{\A v_h,\xi_h}_{L^2(\Omega)} &= \skp{\eta_1 \A \psi_h,\xi_h}_{L^2(\Omega)} \quad \forall \xi_h \in S^{1,1}_0(\T_h), \\
\label{eqwh}
  \skp{\A w_h,\xi_h}_{L^2(\Omega)} &= \skp{(1-\eta_1) \A \psi_h,\xi_h}_{L^2(\Omega)} \quad \forall \xi_h \in S^{1,1}_0(\T_h).
\end{align}
These definitions and the unique solvability of the Galerkin formulation indeed give $\psi_h = v_h + w_h$, and 
we call this the {\bf near-field $(v_h)$/far-field $(w_h)$ splitting of $\psi_h$}.
With the triangle inequality, we write
\begin{align}\label{eq:splitting}
  {\large\textcircled{\normalsize \text{3}}} = \norm{\eta_0 \psi_h}_{H^1(\Omega)} \lesssim  \norm{\eta_0v_h}_{H^1(\Omega)} + \norm{\eta_0 w_h}_{H^1(\Omega)}.
\end{align}
{\bf Preliminary estimates:} To estimate the near-field and far-field, we prove the following estimates for the $L^2$-projection
\begin{align}
\label{eq:Hs-phi-pi-phi} 
\|\eta_2 (\phi - \Pi^{L^2} \phi) \|_{\widetilde{H}^s(\Omega)} & \lesssim h^{1-s}_1 \left( \|\eta_4 e\|_{L^2(\Omega)} + \|e\|_{H^{s-1/2}(\Omega)} \right), \\
\label{eq:Hs-1-phi-pi-phi} 
\|\phi - \Pi^{L^2} \phi\|_{\widetilde{H}^{s-1}(\Omega)} & \lesssim h\|\phi\|_{\widetilde{H}^s(\Omega)} \lesssim h \|e\|_{H^{s-1/2}(\Omega)}, \\
\label{eq:local-Aphi-pi-phi} 
\|\eta_2 {\mathcal A} (\phi - \Pi^{L^2}\phi) \|_{H^{-s}(\Omega)} &\lesssim h_1^{1-s} \|\eta_4 e\|_{L^2(\Omega)} +  (h_1^{1-s}+h) \|e\|_{H^{s-1/2}(\Omega)},
\end{align}
and the discrete function $\psi_h$
\begin{align}
\label{eq:local-phih} 
\|\eta_2 \psi_h\|_{\widetilde{H}^s(\Omega)} &\lesssim \|\eta_2 e\|_{\widetilde{H}^s(\Omega)} + h^{1-s}_1 \|\eta_4 e\|_{L^2(\Omega)} + \|e\|_{H^{s-1/2}(\Omega)}, \\
\label{eq:local-Aphih} 
\|\eta_2 {\mathcal A} \psi_h\|_{H^{-s}(\Omega)} &\lesssim \|\eta_4 e\|_{\widetilde{H}^s(\Omega)} + \|e\|_{H^{s-1/2}(\Omega)}.
\end{align} 
\emph{Proof of (\ref{eq:Hs-phi-pi-phi}):}  
Applying \eqref{eq:L2projectionApprox} with $D_0 = \Omega_{3/10}$ 
and $D_1 = \Omega_{4/10}$, which satisfy $\dist(D_0,\partial D_1) = R/10 \geq 4h_1$, we obtain
\begin{align*}
  \norm{\eta_2(\Pi^{L^2}\phi-\phi)}_{\widetilde{H}^{s}(\Omega)} &\lesssim h^{1-s}_1\norm{\phi}_{H^1(\Omega_{4/10})} + 
  e^{-c/h_1}\norm{\phi}_{L^2(\Omega)}\\
&\leq h^{1-s}_1\norm{\eta_4 \phi}_{H^1(\Omega)} + 
e^{-c/h_1}\norm{\phi}_{\widetilde H^s(\Omega)}\\
&\!\!\!\! \!\!\!\!\!\! \stackrel{(\ref{eq:estphih}),(\ref{eq:estphihglob})}{\lesssim} \!\! h^{1-s}_1 \left( \norm{e}_{H^{s-1/2}(\Omega)} + \norm{\eta_4 e}_{L^2(\Omega)} \right) + e^{-c/h_1}\norm{e}_{H^{s-1/2}(\Omega)} .
\end{align*}
Bounding $e^{-c/h_1} \lesssim h^{1-s}_1$ completes the proof
of (\ref{eq:Hs-phi-pi-phi}). 
\newline 
\emph{Proof of (\ref{eq:Hs-1-phi-pi-phi}):}  
For $s>1/2$, we obtain from Lemma~\ref{lem:L2} 
\begin{align*}
\norm{\Pi^{L^2}\phi-\phi}_{\widetilde H^{s-1}(\Omega)} \stackrel{\text{Lem.~\ref{lem:L2}}}{\lesssim} h  \norm{\phi}_{\widetilde H^s(\Omega)} \stackrel{\eqref{eq:estphihglob}}{\lesssim} h\norm{e}_{H^{s-1/2}(\Omega)}.
\end{align*}
For the case $s\leq 1/2$, we have $\phi \in \widetilde H^{1/2+s-\varepsilon}(\Omega)$ with $\varepsilon \in (0,s/2)$ given by
Assumption~\ref{ass:shift}. Together with Lemma~\ref{lem:L2} and \eqref{eq:estphihglob} this gives 
\begin{align*}
  \norm{\Pi^{L^2}\phi-\phi}_{\widetilde H^{s-1}(\Omega)} 
  \lesssim h \norm{\phi}_{\widetilde H^{1/2+s-\varepsilon}(\Omega)} 
  \lesssim h\norm{e}_{H^{s-1/2}(\Omega)}.
\end{align*}
\emph{Proof of~\eqref{eq:local-Aphi-pi-phi}:} The definition of commutators, ellipticity of $\mathcal{A}$, and Lemma~\ref{lem:commutator} show
\begin{align*}
  \|\eta_2 {\mathcal A} (\phi - \Pi^{L^2}\phi)\|_{H^{-s}(\Omega)}
  &\leq \|{\mathcal A}(\eta_2 (\phi - \Pi^{L^2} \phi))\|_{H^{-s}(\Omega)} + \|{\mathcal C}_{\eta_2} (\phi - \Pi^{L^2} \phi)\|_{H^{-s}(\Omega)}\\
  &\lesssim \|\eta_2 (\phi - \Pi^{L^2}\phi)\|_{\widetilde{H}^s(\Omega)} + \|\phi - \Pi^{L^2}\phi\|_{\widetilde{H}^{s-1}(\Omega)}\\
  &\stackrel{\eqref{eq:Hs-phi-pi-phi},\eqref{eq:Hs-1-phi-pi-phi}}{\lesssim}h_1^{1-s} \|\eta_4 e\|_{L^2(\Omega)} +  (h_1^{1-s} + h) \|e\|_{H^{s-1/2}(\Omega)}.
\end{align*}
\emph{Proof of (\ref{eq:local-phih}):} The definition of $\psi_h$ and stability of the Galerkin projection imply
\begin{align*}
\|\eta_2 \psi_h\|_{\widetilde{H}^s(\Omega)} &\leq \|\eta_2 \Pi (\eta_4 e)\|_{\widetilde{H}^s(\Omega)} + \|\eta_2 (\phi - \Pi^{L^2}\phi)\|_{\widetilde{H}^s(\Omega)} 
+ \|\eta_2 \phi\|_{\widetilde{H}^s(\Omega)}, \\
&\stackrel{(\ref{eq:Hs-phi-pi-phi}), (\ref{eq:estphihglob})}{\lesssim} \|\eta_4 e\|_{\widetilde{H}^s(\Omega)} + h^{1-s}_1 \|\eta_4 e\|_{L^2(\Omega)} + 
\|e\|_{H^{s-1/2}(\Omega)}. 
\end{align*}
\emph{Proof of (\ref{eq:local-Aphih}):} Ellipticity of $\mathcal{A}$, stability of the Galerkin projection, Lemma~\ref{lem:commutator},
$\| \cdot \|_{L^2(\Omega)} \leq \| \cdot \|_{\widetilde H^s(\Omega)}$ and $h,h_1\lesssim 1$ imply
\begin{align*}
 \|\eta_2 {\mathcal A} \psi_h\|_{H^{-s}(\Omega)} & \leq 
\|\eta_2 {\mathcal A} \Pi (\eta_4 e)\|_{H^{-s}(\Omega)} 
+ \|\eta_2 {\mathcal A} (\phi - \Pi^{L^2}\phi)\|_{H^{-s}(\Omega)} 
+ \|\eta_2 {\mathcal A} \phi\|_{H^{-s}(\Omega)} \\
&\lesssim \| \eta_4 e\|_{\widetilde{H}^s(\Omega)} + \|\eta_2 {\mathcal A} (\phi - \Pi^{L^2}\phi)\|_{H^{-s}(\Omega)}
+ \|\phi\|_{\widetilde{H}^{s}(\Omega)}\\
&\stackrel{\eqref{eq:local-Aphi-pi-phi},\eqref{eq:estphihglob}}{\lesssim} \|\eta_4 e\|_{\widetilde{H}^s(\Omega)} + h_1^{1-s} \|\eta_4 e\|_{L^2(\Omega)} + \|e\|_{H^{s-1/2}(\Omega)}\\
&\lesssim \|\eta_4 e\|_{\widetilde{H}^s(\Omega)} + \|e\|_{H^{s-1/2}(\Omega)}.
\end{align*}
We also note that $\psi_h$ satisfies
a useful equation: For $\xi_h \in S^{1,1}_0(\T_h)$ with $\supp \xi_h \subset \overline{\omega_{\eta_1}}\subset \Omega_{2/10}$, 
the Galerkin orthogonality and $\eta_4 \equiv 1$ on $\omega_{\eta_1}$ imply
\begin{align}\label{eq:ortho}
0 &= \skp{\A e,\xi_h}_{L^2(\Omega)} = \skp{\A e,\eta_4 \xi_h}_{L^2(\Omega)} = \skp{\A(\eta_4 e)-\C_{\eta_4}e, \xi_h}_{L^2(\Omega)}\nonumber \\ &=  
\skp{\A(\Pi(\eta_4 e))-\eta_4\C_{\eta_4}e, \xi_h}_{L^2(\Omega)} =
 \skp{\A(\Pi(\eta_4 e)-\A^{-1}(\eta_4\C_{\eta_4}e)), \xi_h}_{L^2(\Omega)} \nonumber \\ &=  
  \skp{\A\psi_h, \xi_h}_{L^2(\Omega)} +  \skp{\A(\Pi^{L^2}\phi-\phi), \xi_h}_{L^2(\Omega)} .
\end{align}

{\bf Estimate of the near-field $\|\eta_0 v_h\|_{\widetilde{H}^1(\Omega)}$:} We exploit locality properties of the near-field to show the estimate
\begin{align}\label{eq:NFest}
  \norm{\eta_0v_h}_{H^1(\Omega)} \lesssim  h_1^{s} \norm{\eta_4 e}_{\widetilde H^s(\Omega)} + \|\eta_4 e\|_{L^2(\Omega)}  + (h/h_1)^{1-s} \norm{e}_{H^{s-1/2}(\Omega)}. 
\end{align}
To see \eqref{eq:NFest}, we start with the inverse inequality~\eqref{invest:a}
\begin{align}\label{eq:NFest-a}
  \norm{\eta_0v_h}_{H^1(\Omega)} \lesssim h_1^{s-1} \norm{v_h}_{H^s(\Omega_{1/10})} \lesssim h_1^{s-1} \|v_h\|_{\widetilde{H}^s(\Omega)}.
\end{align}
The mapping property and ellipticity of $\A$, the definition of $v_h$, and \eqref{eq:ortho} 
imply
\begin{align*}
 &\norm{v_h}_{\widetilde H^s(\Omega)} = \sup_{w \in H^{-s}(\Omega)}\frac{\skp{v_h,w}_{L^2(\Omega)}}{\norm{w}_{H^{-s}(\Omega)}} \lesssim
 \sup_{\varphi \in \widetilde H^s(\Omega)}\frac{\skp{v_h,\A \varphi}_{L^2(\Omega)}}{\norm{\varphi}_{\widetilde H^{s}(\Omega)}}  \nonumber \\
&\qquad=   \sup_{\varphi \in \widetilde H^s(\Omega)}\frac{\skp{v_h,\A \Pi\varphi}_{L^2(\Omega)}}{\norm{\varphi}_{\widetilde H^{s}(\Omega)}}
= \sup_{\varphi \in \widetilde H^s(\Omega)}\frac{\skp{\eta_1 \A \psi_h, \Pi\varphi}_{L^2(\Omega)}}{\norm{\varphi}_{\widetilde H^{s}(\Omega)}} \nonumber  \\
&\qquad \stackrel{\eqref{eq:ortho}}{=} \sup_{\varphi \in \widetilde H^s(\Omega)}\frac{\skp{\A \psi_h, \eta_1\Pi\varphi-\mathcal{J}_h(\eta_1\Pi\varphi)}_{L^2(\Omega)}
-\skp{\A(\Pi^{L^2}\phi-\phi), \mathcal{J}_h(\eta_1\Pi\varphi)}_{L^2(\Omega)}}{\norm{\varphi}_{\widetilde H^{s}(\Omega)}}.
 \end{align*}
 For the first term in the numerator, we use~\eqref{eq:SZ:supp}, in particular $\eta_2 \equiv 1$ on
 $\operatorname{supp} (\eta_1 \Pi \varphi - {\mathcal J}_h (\eta_1 \Pi \varphi))$, and
we may estimate using the bound~\eqref{eq:SZ} and $\| \cdot \|_{H^s(\Omega_{3/10})}\lesssim \| \cdot \|_{\widetilde H^s(\Omega)}$
 \begin{align*}
  \abs{\skp{\A \psi_h, \eta_1\Pi\varphi-\mathcal{J}_h(\eta_1\Pi\varphi)}_{L^2(\Omega)}} &\lesssim
  \norm{\eta_2 \A \psi_h}_{H^{-s}(\Omega)}h_1\norm{\Pi \varphi}_{\widetilde H^s(\Omega)} \\ &\stackrel{(\ref{eq:local-Aphih})}{\lesssim }
  h_1 \left(\|\eta_4 e\|_{\widetilde{H}^s(\Omega)} + \|e\|_{H^{s-1/2}(\Omega)} \right) \norm{\varphi}_{\widetilde H^s(\Omega)}.
 \end{align*}
For the second term in the numerator, we again use that $\eta_2 \equiv 1$ on $\operatorname{supp} {\mathcal J}_h (\eta_1 \Pi \varphi)$ to estimate  
with (\ref{eq:local-Aphi-pi-phi}) and the stability of the operator ${\mathcal J}_h$
\begin{align*}
 &  \abs{\skp{\A(\Pi^{L^2}\phi-\phi), \mathcal{J}_h(\eta_1\Pi\varphi)}_{L^2(\Omega)}} 
\lesssim \|\eta_2 {\mathcal A} (\phi - \Pi^{L^2}\phi)\|_{H^{-s}(\Omega)} \|{\mathcal J}_h (\eta_1 \Pi\varphi)\|_{\widetilde{H}^s(\Omega)} \\
&\qquad\qquad\qquad\stackrel{(\ref{eq:local-Aphi-pi-phi})}{\lesssim} \left(h_1^{1-s} \|\eta_4 e\|_{L^2(\Omega)} + (h_1^{1-s} +h) \|e\|_{H^{s-1/2}(\Omega)}\right) \|\varphi\|_{\widetilde{H}^s(\Omega)}. 
\end{align*}
Combining the previous three formulas, we obtain
\begin{equation}
\label{eq:estimate-vh}
\|v_h\|_{\widetilde{H}^s(\Omega)} \lesssim h_1 \|\eta_4 e\|_{\widetilde{H}^s(\Omega)} + h_1^{1-s} \|\eta_4 e\|_{L^2(\Omega)} + (h_1^{1-s} + h) \|e\|_{H^{s-1/2}(\Omega)}. 
\end{equation}
Inserting (\ref{eq:estimate-vh}) in (\ref{eq:NFest-a}) yields (\ref{eq:NFest}) since $h_1 \leq h$. 

{\bf Estimate of the far-field $\|\eta_0 w_h\|_{\widetilde{H}^1(\Omega)}$:}
We remark that the triangle inequality gives 
$\|\eta_2 w_h\|_{\widetilde{H}^s(\Omega)} \leq \|\eta_2 \psi_h \|_{\widetilde{H}^s(\Omega)} + \|v_h\|_{\widetilde{H}^s(\Omega)} $ 
and therefore the estimate for the near-field provides
\begin{align}
\nonumber 
\|\eta_2 w_h\|_{\widetilde{H}^s(\Omega)} 
& \stackrel{(\ref{eq:local-phih}), (\ref{eq:estimate-vh})}{\lesssim} 
\|\eta_2 e\|_{\widetilde{H}^s(\Omega)} + h_1 \|\eta_4 e\|_{\widetilde{H}^s(\Omega)} + h_1^{1-s}\|\eta_4 e\|_{L^2(\Omega)} + \|e\|_{H^{s-1/2}(\Omega)}  \\
\label{eq:est-wh}
&\lesssim \|\eta_4 e\|_{\widetilde{H}^s(\Omega)} + \|e\|_{H^{s-1/2}(\Omega)}. 
\end{align}
To obtain an estimate for the $H^1$-norm of the far-field, we exploit additional smoothness properties of a suitable Caffarelli-Silvestre extension problem to show the estimate
\begin{align}\label{eq:FFest}
\|\eta_0 w_h\|_{H^1(\Omega)} 
 \lesssim \bigl(1+h_1^{s-1} h^{1-2\varepsilon}\bigr)
 \left( 
   \|\eta_4 e\|_{\widetilde{H}^s(\Omega)} + \|e\|_{H^{s-1/2}(\Omega)}
 \right) 
\end{align}
for $\varepsilon > 0$ given by Assumption~\ref{ass:shift}.
Set 
\begin{equation}
\label{eq:w,psi}
w:= {\mathcal A}^{-1} \left((1-\eta_1){\mathcal A}\psi_h\right) \in \widetilde{H}^s(\Omega), 
\qquad \psi:= \Pi (\eta_4 e) - \phi \in \widetilde{H}^s(\Omega). 
\end{equation}
These definitions directly imply
\begin{align}\label{eq:rep-of-w}
w_h = \Pi w \qquad \text{ and } \qquad  w & = (1-\eta_1) \psi_h  + {\mathcal A}^{-1} {\mathcal C}_{\eta_1} \psi_h, 
\end{align}
and we claim the following two assertions: 
\begin{align}
\label{eq:Cpsi}
\|{\mathcal A}^{-1} {\mathcal C}_{\eta_1} \psi\|_{\widetilde{H}^{s+1/2-\varepsilon}(\Omega)} 
&\lesssim \|\eta_1 e\|_{\widetilde{H}^s(\Omega)} + \|e\|_{H^{s-1/2}(\Omega)}, \\
\label{eq:local-w}
\|\eta_2 w\|_{\widetilde{H}^s(\Omega)} & \lesssim \|\eta_4 e\|_{\widetilde{H}^s(\Omega)} + \|e\|_{H^{s-1/2}(\Omega)}. 
\end{align}
\emph{Proof of (\ref{eq:Cpsi}):} We use Assumption~\ref{ass:shift} and Lemma~\ref{lem:commutator} to get 
\begin{align*}
\|{\mathcal A}^{-1}{\mathcal C}_{\eta_1} \psi\|_{H^{s+1/2-\varepsilon}(\Omega)} &\lesssim 
\|{\mathcal C}_{\eta_1} \psi\|_{H^{-s+1/2-\varepsilon}(\Omega)} 
 \lesssim 
\|\psi\|_{\widetilde H^{s-1/2-\varepsilon}(\Omega)}  \\
& \leq \|\Pi (\eta_4 e)\|_{\widetilde{H}^s(\Omega)} + \|\phi\|_{\widetilde{H}^s(\Omega)} 
\stackrel{(\ref{eq:estphihglob})}{\lesssim}  \|\eta_4 e\|_{\widetilde{H}^s(\Omega)} + \|e\|_{H^{s-1/2}(\Omega)}. 
\end{align*}
\emph{Proof of (\ref{eq:local-w}):} Using (\ref{eq:rep-of-w}) and the definition of $\psi_h$, we write 
\begin{align*}
\eta_2 w & = \eta_2 (1-\eta_1) \psi_h + \eta_2 {\mathcal A}^{-1} {\mathcal C}_{\eta_1} \psi + \eta_2 {\mathcal A}^{-1} {\mathcal C}_{\eta_1} (\phi - \Pi^{L^2} \phi). 
\end{align*}
With the triangle inequality and the fact that
$\|\eta_2\eta_1\psi_h \|_{\widetilde H^s(\Omega)} \lesssim \|\eta_2\psi_h \|_{\widetilde H^s(\Omega)}$, the first term can be estimated 
using~\eqref{eq:local-phih}. The second term is estimated by~\eqref{eq:Cpsi}, using $s \leq s+1/2-\varepsilon$. To bound the third term,
we use the ellipticity of $\mathcal{A}$, Lemma~\ref{lem:commutator}, and~\eqref{eq:Hs-1-phi-pi-phi}. Then, using
$\| \eta_2 e\|_{\widetilde H^s(\Omega)}\lesssim \| \eta_4 e \|_{\widetilde H^s(\Omega)}$ and $h\lesssim 1$, we obtain (\ref{eq:local-w}). 
\bigskip

We continue with the proof of~\eqref{eq:FFest}.
Using $\eta_1\equiv 1$ on $\supp\eta_0$ and the triangle inequality, we get 
\begin{align*}
& \norm{\eta_0 w_h}_{H^1(\Omega)} =
\norm{\eta_1\eta_0 w_h}_{H^1(\Omega)}\\
&\quad \quad \leq \norm{\eta_1\left(\eta_0 w_h - \Pi(\eta_0 w_h)\right)}_{H^1(\Omega)} + 
\norm{\eta_1\Pi(\eta_0(w- w_h ))}_{H^1(\Omega)} + \norm{\eta_1\Pi(\eta_0 w)}_{H^1(\Omega)}\\ 
&
  \quad \quad 
=:
{\large\textcircled{\normalsize \text{a}}}+{\large\textcircled{\normalsize \text{b}}}+{\large\textcircled{\normalsize \text{c}}}. 
\end{align*}
The bound~\eqref{eq:Galerkinprojb} from Lemma~\ref{lem:superapprox}, $\eta_1\equiv 1$ on $\Omega_{1/10}$ and $\eta_2 \equiv 1$ on $\Omega_{2/10}$, 
together with the inverse estimate~\eqref{invest:a} give
\begin{align*}
  {\large\textcircled{\normalsize \text{a}}} &= \norm{\eta_1\left(\eta_0 w_h - \Pi(\eta_0 w_h)\right)}_{H^1(\Omega)}
  \lesssim h_1 \norm{w_h}_{H^1(\Omega_{1/10})} \leq h_1 \norm{\eta_1w_h}_{H^1(\Omega)}\\
  &\stackrel{(\ref{invest:a})}{\lesssim} h_1^{s} \norm{w_h}_{H^s(\Omega_{2/10})} \lesssim h_1^s \|\eta_2 w_h\|_{\widetilde{H}^s(\Omega)} 
\stackrel{(\ref{eq:est-wh})}{\lesssim} h_1^s \left( \|\eta_4 e\|_{\widetilde{H}^s(\Omega)} + \|e\|_{H^{s-1/2}(\Omega)}\right). 
\end{align*}
To bound ${\large\textcircled{\normalsize \text{b}}}$ we write $e_w := w-w_h$ and apply
the inverse estimate~\eqref{invest:a} 
to obtain 
\begin{align}\label{eq:estb}
  {\large\textcircled{\normalsize \text{b}}} = \norm{\eta_1\Pi(\eta_0(w- w_h ))}_{H^1(\Omega)}
\lesssim h_1^{s-1} \norm{\Pi(\eta_0 e_w)}_{\widetilde H^{s}(\Omega)}. 
\end{align}
The first identity in~\eqref{eq:rep-of-w}, Galerkin orthogonality, ellipticity of $\A$, 
the definition of the commutator $\C_{\eta_0}$, the bound~\eqref{eq:SZ},  
and $\| \cdot \|_{H^s(\Omega_{2/10})}\lesssim \| \cdot \|_{\widetilde H^s(\Omega)}$
lead to
\begin{align*}
\norm{\Pi(\eta_0 e_w)}_{\widetilde{H}^s(\Omega)}^2&\lesssim
\skp{\A(\Pi(\eta_0 e_w)),\Pi(\eta_0 e_w)}_{L^2(\Omega)} = \skp{\A(\eta_0 e_w),\Pi(\eta_0 e_w)}_{L^2(\Omega)} \\
 &=\skp{\eta_0 \A e_w + \C_{\eta_0} e_w,\Pi(\eta_0 e_w)}_{L^2(\Omega)} \\&= 
\skp{\A e_w ,\eta_0\Pi(\eta_0 e_w)-\mathcal{J}_h(\eta_0\Pi(\eta_0 e_w))}_{L^2(\Omega)} + 
\skp{\C_{\eta_0} e_w,\Pi(\eta_0 e_w)}_{L^2(\Omega)} \\
&\lesssim h_1\norm{\A e_w}_{H^{-s}(\Omega)} \norm{\Pi(\eta_0 e_w)}_{\widetilde  H^s(\Omega)} + 
\norm{\C_{\eta_0}e_w}_{H^{-s}(\Omega)}\norm{\Pi(\eta_0 e_w)}_{\widetilde H^s(\Omega)} 
\\&\stackrel{Lem.~\ref{lem:commutator}}{\lesssim} \norm{\Pi(\eta_0 e_w)}_{\widetilde H^s(\Omega)}\left(h_1\norm{e_w}_{\widetilde H^{s}(\Omega)}
+\norm{e_w}_{\widetilde H^{s-1}(\Omega)}\right).
\end{align*}
Next, we estimate $\|e_w\|_{\widetilde{H}^s(\Omega)} = \|w - \Pi w\|_{\widetilde{H}^s(\Omega)}$.
Note that
\begin{align*}
  w-\Pi w = w-(1-\eta_1)\psi_h - \Pi (w-(1-\eta_1)\psi_h) + \eta_1\psi_h - \Pi (\eta_1\psi_h).
\end{align*}
Hence, triangle inequality and quasioptimality of the Galerkin projection $\Pi$ give 
\begin{align*}
\|e_w\|_{\widetilde{H}^s(\Omega)} & \leq \underbrace{ \inf_{z_h \in S^{1,1}_0(\T_h)} \|w - (1-\eta_1) \psi_h - z_h\|_{\widetilde{H}^s(\Omega)} }_{=:T_1} 
+ \underbrace{
\inf_{z'_h \in S^{1,1}_0(\T_h)} \|\eta_1 \psi_h - z'_h\|_{\widetilde{H}^s(\Omega)} }_{=:T_2}.
\end{align*}
The term $T_2$ is controlled with (\ref{eq:SZ}) and gives 
\begin{align*}
T_2 
\stackrel{(\ref{eq:SZ})}{\lesssim} h_1 \|\eta_2 \psi_h \|_{\widetilde{H}^s(\Omega)} 
\stackrel{(\ref{eq:local-phih})}{\lesssim} h_1 \left( \|\eta_2 e\|_{\widetilde{H}^s(\Omega)} + h^{1-s}_1 \|\eta_4 e\|_{L^2(\Omega)} + \|e\|_{H^{s-1/2}(\Omega)}\right). 
\end{align*}
For $T_1$, we use (\ref{eq:rep-of-w}) to get 
$w - (1-\eta_1) \psi_h = {\mathcal A}^{-1} {\mathcal C}_{\eta_1} \psi_h 
= {\mathcal A}^{-1} {\mathcal C}_{\eta_1} \psi + {\mathcal A}^{-1}{\mathcal C}_{\eta_1} (\psi_h - \psi)
= {\mathcal A}^{-1} {\mathcal C}_{\eta_1} \psi + {\mathcal A}^{-1}{\mathcal C}_{\eta_1} (\Pi^{L^2} \phi - \phi)$ and then estimate 
\begin{align*}
T_1 
& \lesssim 
\inf_{z_h \in S^{1,1}_0(\T_h)} \|{\mathcal A}^{-1}{\mathcal C}_{\eta_1} \psi - z_h\|_{\widetilde{H}^s(\Omega)} + \|{\mathcal C}_{\eta_1} (\phi -\Pi^{L^2} \phi)\|_{H^{-s}(\Omega)} \\
& \lesssim h^{1/2-\varepsilon} \|{\mathcal A}^{-1} {\mathcal C}_{\eta_1} \psi\|_{\widetilde{H}^{s+1/2-\varepsilon}(\Omega)} + \|\phi - \Pi^{L^2}\phi\|_{\widetilde{H}^{s-1}(\Omega)} \\
&\stackrel{(\ref{eq:Cpsi}), (\ref{eq:Hs-1-phi-pi-phi})}{\lesssim } h^{1/2-\varepsilon}  \left(\|\eta_1 e\|_{\widetilde{H}^s(\Omega)} + \|e\|_{H^{s-1/2}(\Omega)}\right) + 
h \|e\|_{H^{s-1/2}(\Omega)}. 
\end{align*}
We conclude using $h_1 \leq h \lesssim 1$
\begin{equation}
\label{eq:energy-estimate-ew}
\|e_w\|_{\widetilde{H}^s(\Omega)} \lesssim h^{1/2-\varepsilon} \|\eta_4 e\|_{\widetilde{H}^s(\Omega)} 
+ h^{1/2-\varepsilon}\|e\|_{H^{s-1/2}(\Omega)}.
\end{equation}
For $\|e_w\|_{\widetilde{H}^{s-1}(\Omega)}$ we employ a standard duality argument, noting that the maximal exploitable regularity is $\widetilde H^{1/2+s-\varepsilon}(\Omega)$ to obtain
\begin{align*}
 \norm{e_w}_{\widetilde H^{s-1}(\Omega)} 
& \lesssim  
 \norm{e_w}_{\widetilde H^{-1/2+s+\varepsilon}(\Omega)} \lesssim h^{1/2-\varepsilon} \norm{e_w}_{\widetilde H^s(\Omega)}.
\end{align*}
Inserting everything into \eqref{eq:estb} leads to
\begin{align*}
  {\large\textcircled{\normalsize \text{b}}} \lesssim 
 h_1^{s-1} h^{1-2\varepsilon} 
 \left( 
   \|\eta_4 e\|_{\widetilde{H}^s(\Omega)} + \|e\|_{H^{s-1/2}(\Omega)}
   \right).
\end{align*}
To estimate {\large\textcircled{\normalsize \text{c}}},
we denote by $U_{\rm far}$ the solution of the extension problem \eqref{eq:extension} with data $\operatorname{tr}U_{\rm far} = \eta_2 w$.
We apply a variation of \cite[Lemma~4.3]{FMP19} with open sets $B_x \times B_y =: B \subset B' := B_x' \times B_y'$
and an extension $\widehat{\eta_0}$ of the cut-off function $\eta_0$ satisfying $\supp \widehat{\eta_0} \subset B'$. 
Instead of prescribing support properties of $\operatorname*{tr} U_{\rm far}$ as in \cite[Lemma~4.3]{FMP19}, an inspection of
the proof therein shows that 
choosing the same test function $V$ in the weak formulation of \eqref{eq:extension} for the difference quotient argument works,
as long as $\operatorname*{tr} V \cdot (\mathcal{Y}^\alpha \partial_{\mathcal{Y}}U_{\rm far})|_{y=0} \equiv 0$.
The said test function $V$ is a second order difference quotient of $\widehat{\eta_0}U_{\rm far}$ and therefore its
trace is supported in $\Omega_{1/10}$. Since 
\begin{align*}
-d_s(\mathcal{Y}^\alpha \partial_{\mathcal{Y}}U_{\rm far})|_{y=0} = \mathcal{A} (\eta_2 w),
\end{align*}
the assumption $\eta_1 \equiv 1$ on $\Omega_{1/10}$ indeed gives $\operatorname*{tr} V \cdot (\mathcal{Y}^\alpha \partial_{\mathcal{Y}}U_{\rm far})|_{y=0} \equiv 0$ and therefore the arguments of \cite[Lemma~4.3]{FMP19} imply
\begin{align}\label{eq:CaccCS}
 \norm{D_x(\nabla U_{\rm far})}_{L^2_\alpha(B)} \lesssim \norm{U_{\rm far}}_{H^1_\alpha(B')}.
\end{align}
The $H^1$-stability~\eqref{eq:Galerkinproja} of the Galerkin projection from Lemma~\ref{lem:superapprox}, the multiplicative trace inequality from \cite[Lemma~3.7]{KarMel18}, and the Lax-Milgram Lemma then give 
\begin{align*}
  \large\textcircled{\normalsize \text{c}} &= \norm{\eta_1\Pi(\eta_0 w)}_{H^1(\Omega)}\lesssim   \norm{\eta_0 w}_{H^1(\Omega)} \lesssim  \abs{\eta_0 w}_{H^1(\Omega)} + \norm{\eta_0 w}_{L^2(\Omega)} \\ &\lesssim
 \norm{\nabla_x U_{\rm far}}_{L^2_\alpha(B)} +   \norm{\nabla_x U_{\rm far}}_{L^2_\alpha(B)}^{(1-\alpha)/2}
 \norm{\partial_{\mathcal{Y}}\nabla_x U_{\rm far}}_{L^2_\alpha(B)}^{(1+\alpha)/2}+ \norm{U_{\rm far}}_{H^1_\alpha(B)} \\ 
 &\stackrel{\eqref{eq:CaccCS}}{\lesssim}  \norm{U_{\rm far}}_{H^1_\alpha(B')} \lesssim \norm{\eta_2 w}_{\widetilde H^s(\Omega)} 
\stackrel{(\ref{eq:local-w})}  {\lesssim} \|\eta_4 e\|_{\widetilde{H}^s(\Omega)} + \|e\|_{H^{s-1/2}(\Omega)}. 
\end{align*}
Combining the estimates for {\large\textcircled{\normalsize \text{a}}}, {\large\textcircled{\normalsize \text{b}}}, {\large\textcircled{\normalsize \text{c}}} 
gives the stated bound (\ref{eq:FFest}) for $\|\eta_0 w_h\|_{H^1(\Omega)}$. 

Putting together the estimates for
{\large\textcircled{\normalsize \text{1}}}, {\large\textcircled{\normalsize \text{2}}}, and {\large\textcircled{\normalsize \text{3}}} 
and observing that $2 \varepsilon < s$ and $h_1 \leq h$, we arrive at 
\begin{align}
\label{eq:H1-estimate-prelim}
  & \abs{\eta_0 e}_{H^1(\Omega)} \lesssim \\
\nonumber 
 & \qquad \ \norm{u}_{H^1(\Omega_{5/10})} \!+\! h_1\| e \|_{H^1(\Omega_{5/10})}\! +\! (h/h_1)^{1-s} \left[ \norm{\eta_4 e}_{\widetilde{ H}^s(\Omega)}+ \|e\|_{H^{s-1/2}(\Omega)}\right]. 
\end{align}
Let $v_h \in S^{1,1}_0(\T_h)$ be arbitrary. Noting $u - v_h - \Pi(u - v_h) = u - u_h = e$ and applying 
(\ref{eq:H1-estimate-prelim}) to $u - v_h$, we get 
\begin{align}
\label{eq:H1-estimate-prelim-10}
   \abs{\eta_0 e}_{H^1(\Omega)} & \lesssim 
\norm{u-v_h}_{H^1(\Omega_{5/10})} \\
\nonumber 
& \quad \mbox{}+  h_1\|e \|_{H^1(\Omega_{5/10})}\! +\! (h/h_1)^{1-s} \left[ \norm{\eta_4 e}_{\widetilde{ H}^s(\Omega)}+ \|e\|_{H^{s-1/2}(\Omega)}\right].
\end{align}
Using the triangle inequality and inverse estimate \eqref{invest:a} yields 
\begin{align*}
h_1\|e\|_{H^1(\Omega_{5/10})} & \leq h_1\|u- v_h\|_{H^1(\Omega_{5/10})} + h_1\|u_h - v_h\|_{H^1(\Omega_{5/10})} \\
& \lesssim h_1\|u- v_h\|_{H^1(\Omega_{5/10})} + h_1^{s} \|u_h - v_h\|_{H^{s}(\Omega_{6/10})} \\ 
& \lesssim  h_1^s\|u- v_h\|_{H^1(\Omega_{1})} + h_1^{s} \|\eta_6 e\|_{H^{s}(\Omega)}. 
\end{align*}
Inserting this in (\ref{eq:H1-estimate-prelim-10}) and using the estimate of Theorem~\ref{thm:local}(\ref{item:thm:local-i}) (with $\eta_0$ replaced by $\eta_6$ and 
correspondingly $\eta_2$ on the right-hand side replaced with $\eta_8$) yields the claimed statement since $\operatorname{supp} \eta_8 \subset \Omega_1$.
\end{proof}

With a classical duality argument and exploiting the (local and global) regularity of $u$, Theorem~\ref{thm:local}
immediately implies Corollary~\ref{cor:local}.

\begin{proof}[Proof of Corollary~\ref{cor:local}]
We only show the second statement of the corollary, the first statement follows with exactly the same arguments.

The assumptions on the local and global regularity directly imply, see, e.g., \cite{AB17}, that
\begin{align*}
 \norm{u-u_h}_{\widetilde H^s(\Omega)} &\lesssim h^{\alpha}\norm{u}_{\widetilde H^{s+\alpha}(\Omega)}, \\
 \inf_{v_h \in S^{1,1}_0(\T_h)} \abs{u-v_h}_{H^1(\Omega_1)} &\lesssim h^{\min\{\beta,1\}}\norm{u}_{H^{1+\beta}(\Omega_2)}.
\end{align*}
Moreover, since $\T_h$ is quasi-uniform, we have that $h/h_1 \lesssim 1$ and the local $H^s$-norm is treated with the first statement of the corollary.
For the global term in Theorem~\ref{thm:local} and arbitrary $\varepsilon>0$ such that Assumption~\ref{ass:shift} holds, we use a duality argument together with Assumption~\ref{ass:shift}, the Galerkin orthogonality and approximation properties of the Scott-Zhang projection $\mathcal{J}_h$ to estimate
\begin{align*}
 &\norm{u-u_h}_{H^{s-1/2}(\Omega)} 
\lesssim \norm{u - u_h}_{\widetilde H^{s-1/2}(\Omega)} 
\lesssim \norm{u - u_h}_{\widetilde H^{s-1/2+\varepsilon}(\Omega)} 
 \\&\qquad\qquad= \sup_{w\in H^{1/2-s-\varepsilon}(\Omega)} \frac{\skp{u-u_h,w}_{L^2(\Omega)}}{\norm{w}_{H^{1/2-s-\varepsilon}(\Omega)}}
 \lesssim \sup_{v\in \widetilde H^{1/2+s-\varepsilon}(\Omega)} \frac{\abs{\skp{u-u_h,\A v}_{L^2(\Omega)}}}{\norm{v}_{\widetilde H^{1/2+s-\varepsilon}(\Omega)}} \\ &\qquad\qquad= 
 \sup_{v\in \widetilde H^{1/2+s-\varepsilon}(\Omega)} \frac{\abs{\skp{\A(u-u_h),v-\mathcal{J}_h v}_{L^2(\Omega)}}}{\norm{v}_{\widetilde H^{1/2+s-\varepsilon}(\Omega)}} 
 \\&\qquad\qquad 
 \lesssim h^{1/2-\varepsilon} \norm{\A(u-u_h)}_{H^{-s}(\Omega)} \lesssim h^{1/2-\varepsilon} \norm{u-u_h}_{\widetilde H^s(\Omega)},
\end{align*} 
which finishes the proof of the corollary.
\end{proof}

\begin{remark}
\label{remk:3.9}
As mentioned in Remark~\ref{rem:graded} and Remark~\ref{rem:gradedH1}, graded meshes can also be employed. 
Therefore, using duality arguments, the rate of convergence locally can also be deduced as in the 
previous proof of Corollary~\ref{cor:local}. Using graded meshes with 
$h_{\rm min}(\Omega) \sim h_{\rm max}(\Omega)^2$, we refer to \cite{AB17} for the global convergence estimate
\begin{align}\label{eq:slushtermgraded}
\norm{u-u_h}_{\widetilde H^s(\Omega)} \lesssim h^{1-\varepsilon}
\end{align}
for arbitrary (small) $\varepsilon>0$, provided the data is sufficiently regular. 
Hence, in the local $H^1$-norm estimate, the slush term is of higher order, and, 
for solutions locally in $H^2$, one obtains convergence $\mathcal{O}(h)$.

However, in the case of the energy norm the best achievable convergence is $\mathcal{O}(h^{2-s})$, for which 
the slush-term needs to be treated more carefully. First, we note that we can actually weaken the 
$H^{s-1/2}(\Omega)$-norm of the slush term even further by employing commutators of higher order in 
the proof of Theorem~\ref{thm:local}(\ref{item:thm:local-i}). Thus, in the duality argument in the proof of Corollary~\ref{cor:local},
the function $w$ is sufficiently regular, such that we can deduce weighted Sobolev regularity (\cite[Eqn.~(3.6)]{BBNOS18}) for the function 
$v$ satisfying $\A v = w$. As in \cite{BBNOS18,BLN20}, this additional regularity gives
\begin{align*}
\norm{v-\mathcal{J}_h v}_{\widetilde H^{s}(\Omega)} \lesssim h^{1-\varepsilon}
\end{align*}
and combining this with \eqref{eq:slushtermgraded} as in the proof of Corollary~\ref{cor:local} again shows that the 
slush term is of higher order.

We also refer to the numerical example on graded meshes in the next section, which confirms the theoretical convergence rates.
\eremk
\end{remark}

\section{Numerical examples}
\label{sec:numerics}

We illustrate our theoretical results of the previous sections with a numerical example in two 
dimensions. 
On the unit ball $\Omega = B_1(0)$ with constant right-hand side 
$f = 2^{2s}\Gamma(1+s)^2$ the exact solution is, 
see, e.g., \cite{BBNOS18}, 
\begin{align*}
 u(x) = (1-\abs{x}^2)_+^s \qquad \text{where} \;\; g_+ = \max\{g,0\}, 
\end{align*}

We choose the subdomain $\Omega_0 \subset B_1(0)$ as a square centered at the origin with sidelength of $0.4$
as depicted in Figure~\ref{fig:mesh}. The exact solution satisfies $u \in \widetilde H^{1/2+s-\varepsilon}(\Omega)$ for any $\varepsilon >0$ as well as $u \in H^2(\Omega_1)$ on every set $\Omega_0 \subset \Omega_1 \subset \Omega$ satisfying 
$\dist(\Omega_1,\partial \Omega) > 0$.

Figures~\ref{fig:error1}--\ref{fig:error3} show global and local errors in the $L^2$-norm and the $H^1$-seminorm.
The discrete solutions are obtained from the MATLAB code \cite{ABB17}, and the errors are computed elementwise using high precision Gaussian quadrature.

\begin{figure}[ht]
\begin{minipage}{.49\linewidth}
\begin{center}
\includegraphics[width=0.99\textwidth]{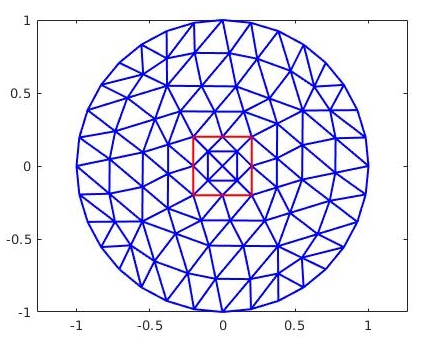}
\vspace{1mm}
\caption{Computational domain, local error computed on red subdomain.}
 \label{fig:mesh}
\end{center}
\end{minipage}
\begin{minipage}{.49\linewidth}
\begin{center}
\includegraphics{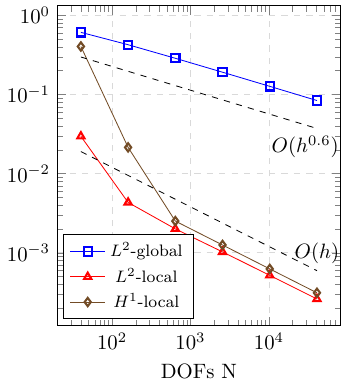}
\end{center}
\caption{Local and global errors in the $L^2$- and $H^1$-norm for $s=0.1$.}
 \label{fig:error1}
\end{minipage}
\end{figure}

As predicted by the theory of \cite{BC19} (for the global $H^1$-error) and Corollary~\ref{cor:local} (for the local $H^1$-error), 
we obtain -- dropping the $\varepsilon > 0$, which we may expect to be arbitrary small in view of the 
shift theorem for smooth $\Omega$ (cf.\ Remark~\ref{remk:shift-theorem})
 --  convergence 
$\mathcal{O}(N^{1/4-s/2}) = \mathcal{O}(h^{s-1/2})$ in the global $H^1$-norm (provided the solutions are in $H^1$), 
$\mathcal{O}(N^{-\min\{1/4+s/2,1/2\}}) = \mathcal{O}(h^{\min\{1/2+s,1\}})$ in the global $L^2$-norm as well as 
$\mathcal{O}(N^{-1/2}) = \mathcal{O}(h)$ both in the local $L^2$-norm 
(see Remark~\ref{rem:sharpness} for the explanation of the limited local convergence rate) and $H^1$-norm. 

\begin{figure}[ht]
\begin{minipage}{.49\linewidth}
\begin{center}
\includegraphics{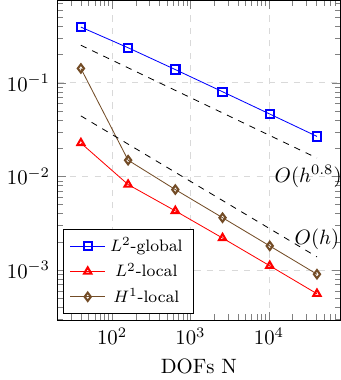}
\end{center}
\end{minipage}
\begin{minipage}{.49\linewidth}
\begin{center}
\includegraphics{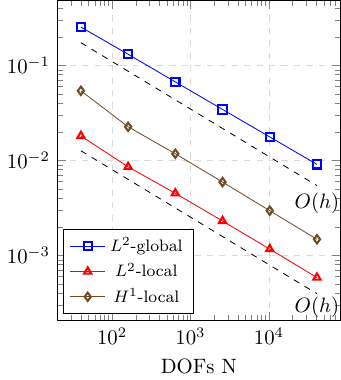}
\end{center}
\end{minipage}
\caption{Local and global errors in the $L^2$- and $H^1$-norm, 
 left: $s=0.3$; 
 right: $s=0.5$.}
 \label{fig:error2}
\end{figure}

\begin{figure}[ht]
\begin{minipage}{.49\linewidth}
\begin{center}
\includegraphics{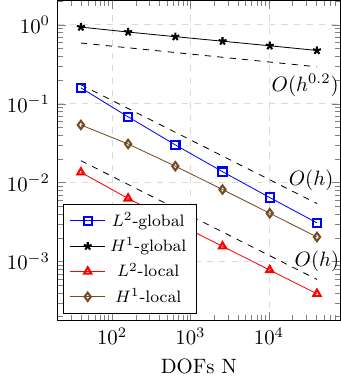}
\end{center}
\end{minipage}
\begin{minipage}{.49\linewidth}
\begin{center}
\includegraphics{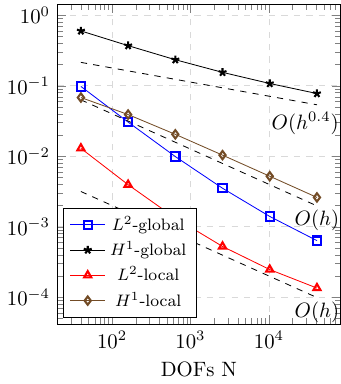}
\end{center}
\end{minipage}
\caption{Local and global errors in the $L^2$- and $H^1$-norm, 
 left: $s=0.7$; 
 right: $s=0.9$.}
 \label{fig:error3}
\end{figure}

Finally, Figure~\ref{fig:error4} computes local and global errors on radially graded meshes, i.e., for a 
fixed mesh size parameter $h>0$, the mesh is graded in such a way that elements
$T \in \mathcal{T}_h$ satisfy $\diam(T) \simeq h^2$ if $\overline{T} \cap \partial \Omega \neq \emptyset$ and 
$\diam(T) \simeq h \dist(T,\partial \Omega)^{1/2}$ otherwise. 
We plot the global error in the energy norm as well as the local errors in $L^2$ and $H^1$ and the upper bound 
$\norm{\cdot}^{1-s}_{L^2}\norm{\cdot}^{s}_{H^1}$ for the local error in the energy norm. 

As expected, see, e.g., \cite{AB17}, one obtains -- up to arbitrary small $\varepsilon>0$ -- convergence 
$\mathcal{O}(h)$ in the global energy norm. For the local errors, the $H^1$-error converges as 
$\mathcal{O}(h)$ and the $L^2$-error is $\mathcal{O}(h^2)$, which produces convergence $\mathcal{O}(h^{2-s})$ 
for the local error in the energy norm, since now the local approximation error is the dominant part.

\begin{figure}[ht]
\begin{minipage}{.49\linewidth}
\begin{center}
\includegraphics{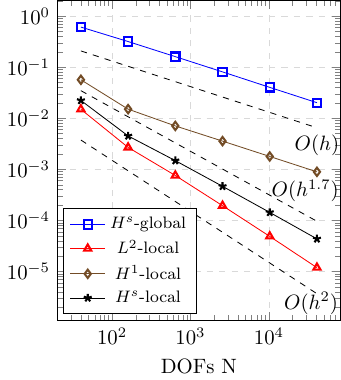}
\end{center}
\end{minipage}
\begin{minipage}{.49\linewidth}
\begin{center}
\includegraphics{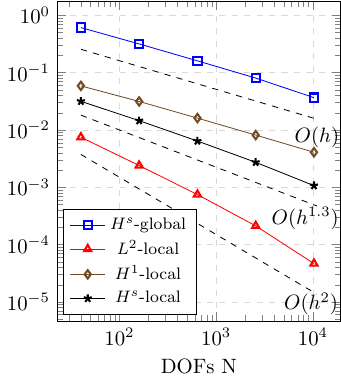}
\end{center}
\end{minipage}
\caption{Local and global errors on graded meshes, 
 left: $s=0.3$; 
 right: $s=0.7$.}
 \label{fig:error4}
\end{figure}

\bibliography{bibliography}{}
\bibliographystyle{amsalpha}
\end{document}